\documentclass[11pt]{amsart}

\usepackage{amsfonts,amsmath,amsthm,amscd,amssymb,latexsym}
\usepackage[T1]{fontenc}
\usepackage{graphicx}
\usepackage{cite}
\usepackage{newtxtext}
\usepackage[varvw]{newtxmath} 

\usepackage{tikz}
\usetikzlibrary{positioning, trees, decorations.markings, patterns, arrows, arrows.meta}
\usepackage{diagmac2}
\usepackage{bm} 

\usepackage[pdftex, colorlinks,
urlcolor=black,
citecolor=black,
linkcolor=blue,
linktocpage=false]%
{hyperref}

\newtheorem{theorem}{Theorem}[section]
\newtheorem{lemma}[theorem]{Lemma}
\newtheorem{corollary}[theorem]{Corollary}
\newtheorem{proposition}[theorem]{Proposition}

\theoremstyle{definition}
\newtheorem{definition}[theorem]{Definition}
\newtheorem{problem}[theorem]{Problem}
\newtheorem{question}[theorem]{Question}
\newtheorem{conjecture}[theorem]{Conjecture}

\theoremstyle{remark}
\newtheorem{example}[theorem]{Example}
\newtheorem{remark}[theorem]{Remark}

\allowdisplaybreaks[4]

\newcommand{\RR}{\mathbb{R}}
\newcommand{\BB}{\mathbf{B}}

\DeclareMathOperator{\diam}{diam}
\DeclareMathOperator{\dist}{dist}

\def\<#1>{\langle #1 \rangle}

\makeatletter
\renewcommand{\p@enumii}{}

\@namedef{subjclassname@2020}{%
  \textup{2020} Mathematics Subject Classification}
\makeatother
\makeatother

\newcommand{\acr}{\newline\indent}

\author{Oleksiy Dovgoshey}
\address{\textbf{Oleksiy Dovgoshey}\acr
Institute of Applied Mathematics and Mechanics
of the NAS of Ukraine, \acr Sloviansk, Ukraine \newline\indent
and Department of Mathematics and Statistics, University of Turku, Finland.}
\email{oleksiy.dovgoshey@gmail.com, oleksiy.dovgoshey@utu.fi}

\title{Hausdorff distance between ultrametric balls}

\subjclass[2020]{Primary 54E35, 54E45}

\keywords{Boundedly compact ultrametric, dense discrete set, locally compact ultrametric, Polish ultrametric, separable ultrametric, ultrametric ball}

\begin{document}

\begin{abstract}
Let \((X, d)\) be an ultrametric space and let \(d_H\) be the Hausdorff distance on the set \(\bar{\BB}_X\) of all closed balls in \((X, d)\). Some interconnections between the properties of the spaces \((X, d)\) and \((\bar{\BB}_X, d_H)\) are described. It is established that the space \((\bar{\BB}_X, d_H)\) has such properties as discreteness, local finiteness, metrical discreteness, completeness, compactness, local compactness if and only if the space \((X, d)\) has these properties. Necessary and sufficient conditions for the separability of the space \((\bar{\BB}_X, d_H)\) are also proved.
\end{abstract}

\maketitle

\section{Introduction}

The main object of our studies is the space \(\bar{\BB}_X\) of closed ultrametric balls endowed with the Hausdorff metric \(d_H\).

The Hausdorff distance was introduced by Felix Hausdorff in~1914 \cite{Hau1914}. Later, Edwards \cite{Edw1975} and, independently, Gromov~\cite{Gro1981PMIHS} expanded Hausdorff construction to the set of isometry classes of compact metric spaces. Ultrametric modifications of Gromov---Hausdorff distance and Hausdorff distance were considered in \cite{Qiu2009pNUAA} and \cite{Qiu2014pNUAA}. A robust alternative to Gromov---Hausdorff distance is the so-called Gromov---Wasserstein distance \cite{AAC+2024TGDBS, Mem2011GDATMATOM, MMWW2023TUGD}.

The Hausdorff distance is closely connected with the so-called Wijsman topology, which was firstly studied by Lechicki and Levi~\cite{LL1987BUMI}. In particular, one can easily proves that, for ultrametric space \((X, d)\), a sequence \((\bar{B}_i)_{i \in \mathbb{N}}\) of closed balls is convergent in \((\bar{\BB}_X, d_H)\) if and only if \((\bar{B}_i)_{i \in \mathbb{N}}\) is convergent in the Wijsman topology. It is interesting to note that, for every metric space \((Y, \delta)\), a sequence of subsets of \(Y\) is convergent with respect to the Vietoris topology generated by \(\delta\) if and only if this sequence is Wijsman convergent for all metrics which are topologically equivalent to \(\delta\) \cite{BNLL1992AMPAIS}. See also \cite{BDD2017BBMSSS} for characterization of Wijsman convergence via Wijsman statistical convergence by moduli.

The interesting applications of the Hausdorff distance to the problem of recognition of an object in an image are given in the book William Rucklidge \cite{Ruc1996EVRUTHD}. Some algorithms for computing the Hausdorff distance between of a model and an image are presented in~\cite{HKR1993CIUTHD}.

The interconnections between the Gurvich---Vyalyi representation \cite{GV2012DAM} of a finite ultrametric space \((X, d)\) and the space \((\bar{\BB}_X, d_H)\) are discussed in \cite{Dov2019pNUAA}. An analog of the Gurvich---Vyalyi representation for totally bounded ultrametric spaces was recently obtained in \cite{Dov2025TBUSALFT}, and it is also closely connected with the Hausdorff distance between the ultrametric balls.

The main goal of this paper is to describe properties common to ultrametric spaces and the corresponding spaces of closed balls equipped with the Hausdorff distance.

The paper is organized as follows. Section~\ref{sec2} contains some necessary definitions and facts from the theory of metric spaces. In particular, in Proposition~\ref{p1.12} and Corollary~\ref{c1.13} we give proofs of two equalities connected with dense discrete subsets of metric spaces.

A brief overview of properties of closed ultrametric balls is given in Section~\ref{sec3}. The construction of the Delhomm\'{e}---Laflamme---Pouzet---Sauer ultrametric space \cite{DLPS2008TaiA} is described in Example~\ref{ex2.6} and this example is systematically used future.

The concept of Hausdorff distance is introduced in Section~\ref{sec4}. The main results of this section are Theorem~\ref{t3.11} and Proposition~\ref{p3.21} which describe the dense discrete subsets of the space \((\bar{\BB}_X, d_H)\) for arbitrary ultrametric space \((X, d)\).

In Section~\ref{sec5} it is proved that each of the following properties of ultrametric spaces \((X, d)\) and \((\bar{\BB}_X, d_H)\) can only be possessed simultaneously:
\begin{itemize}
\item discreteness (Theorem~\ref{t4.1}),
\item local finiteness (Theorem~\ref{t4.3}),
\item metrical discreteness (Theorem~\ref{t4.4}),
\item completeness (Theorem~\ref{t3.13}),
\item total boundedness (Theorem~\ref{t3.16}),
\item compactness (Corollary~\ref{c3.17}),
\item local compactness (Theorem~\ref{t5.14}),
\item compactness of bounded closed subsets (Theorem~\ref{t3.17}).
\end{itemize}
Theorem~\ref{t4.7} gives sufficient conditions under which \((X, d)\) and \((\bar{\BB}_X, d_H)\) are isometric. The necessary and sufficient conditions under which the space \((\bar{\BB}_X, d_H)\) is separable proved in Theorem~\ref{t3.14}. Example~\ref{ex3.19} describes a separable ultrametric space \((X, d)\) with non-separable \((\bar{\BB}_X, d_H)\).

The last Section~\ref{sec6} contains some conjectures related to subject of the paper.

\section{Preliminaries. Metric spaces}
\label{sec2}

Let \(\RR^+\) denote the set \([0, \infty)\) and \(X\) be a nonempty set. A \textit{metric} on $X$ is a function $d\colon X\times X\rightarrow \RR^+$ such that for all \(x\), \(y\), \(z \in X\)
\begin{enumerate}
\item $d(x,y)=d(y,x)$,
\item $(d(x,y)=0)\Leftrightarrow (x=y)$,
\item \(d(x,y)\leq d(x,z) + d(z,y)\).
\end{enumerate}

A metric space \((X, d)\) is \emph{ultrametric} if the \emph{strong triangle inequality}
\[
d(x,y)\leq \max \{d(x,z),d(z,y)\}
\]
holds for all \(x\), \(y\), \(z \in X\). In this case the function \(d\) is called \emph{an ultrametric} on \(X\).

\begin{definition}\label{d2.2}
Let \((X, d)\) and \((Y, \delta)\) be metric spaces. A bijective mapping \(\Phi \colon X \to Y\) is called to be an \emph{isometry} of \((X, d)\) and \((Y, \delta)\) if
\[
d(x,y) = \delta(\Phi(x), \Phi(y))
\]
holds for all \(x\), \(y \in X\). Two metric spaces are \emph{isometric} if there is an isometry of these spaces.
\end{definition}

Let \(A\) be a nonempty subset of a metric space \((X, d)\). The quantity
\begin{equation*}
\diam A :=\sup\{d(x,y) \colon x, y\in A\}
\end{equation*}
is the \emph{diameter} of \(A\). If the inequality \(\diam A < \infty\) holds, then we say that \(A\) is a \emph{bounded} subset of \(X\). If \((X, d)\) is an ultrametric space, then, using the strong triangle inequality, one can easily prove the equality
\begin{equation}\label{e2.3}
\diam A = \sup\{d(p, p_0) \colon p \in A\}
\end{equation}
for every \(p_0 \in A\).

Let \(c\) be a point of a metric space \((X, d)\) and let \(r \geqslant 0\). We define the \emph{closed ball} \(\bar{B}_r(c)\) with the center \(c\) and the radius \(r\) as
\begin{equation}\label{e2.4}
\bar{B}_r(c) := \{x \in X \colon d(c, x) \leqslant r\}.
\end{equation}
Thus, the equality \(\bar{B}_r(c) = \{c\}\) holds if \(r = 0\).

In what follows we will denote by \(\bar{\BB}_{X}\) the set of all closed balls in the metric space \(X\) and by \(\bar{\BB}_{X}^{0}\) the subset of \(\bar{\BB}_{X}\) consisting of all closed balls with strictly positive radii.

\begin{definition}\label{d2.1}
Let \((X, d)\) be a metric space. A \emph{ballean} of the metric space \((X, d)\) is the set \(\bar{\BB}_{X}\).
\end{definition}

\begin{remark}\label{r2.2}
Definition~\ref{d2.1} above is used in papers \cite{Dov2019pNUAA, Dov2020TaAoG}. The term ballean is also used to mean a set \(X\) endowed with a special family of subsets of \(X \times X\). See, for example, \cite{PP2020MS, PB2003, PZ2007}.
\end{remark}

Below we collect some necessary results related to compact metric spaces, totally bounded metric spaces, and complete metric spaces.

\begin{proposition}\label{p1.5}
Let \((X, d)\) be a metric space, let \(Y\) be a nonempty subset of \(X\), and let \(K \subseteq Y \subseteq X\). Then \(K\) is a compact subset of \((Y, d|_{Y \times Y})\) if and only if \(K\) is a compact subset of \((X, d)\).
\end{proposition}

For the proof see, for example, Proposition~5.6 in \cite{Kai2023ACTOMS}.

A proof of the following simple proposition can be found in \cite{Sea2007}, see Theorem~7.9.2.

\begin{proposition}\label{p1.6}
Let \((X, d)\) be a totally bounded metric space and let \(A \subseteq X\) be nonempty. Then the space \((A, d|_{A \times A})\) is also totally bounded.
\end{proposition}

The next statement is the Heine---Borel theorem (see, for example, Theorem~5.26 in~\cite{Kai2023ACTOMS}).

\begin{theorem}\label{t1.11}
A metric space \((X, d)\) is compact if and only if \((X, d)\) is complete and totally bounded.
\end{theorem}

The following theorem can be found, for example, in~\cite{Sea2007} (see Theorem~10.3.2).

\begin{theorem}\label{t1.10}
Let \((X, d)\) be a complete metric space and let \(A \subseteq X\) be nonempty. Then \(A\) is complete if and only if \(A\) is closed in \((X, d)\).
\end{theorem}

A similar result is true for compact metric spaces (see, for example, Theorem~12.2.3 in~\cite{Sea2007}).

\begin{theorem}\label{t1.10.1}
Let \((X, d)\) be a compact metric space and let \(A \subseteq X\) be nonempty. Then \(A\) is compact if and only if \(A\) is closed in \((X, d)\).
\end{theorem}

\begin{remark}\label{r1.7}
Some new results related to compact ultrametric spaces and totally bounded ultrametric spaces can be found in \cite{Dov2024SUPF, DCR2025CUSGBLSG, DK2022LTGCCaDUS, DV2025TBUSGBLR, BDK2022UACMG, DS2022TRoUCaS}.
\end{remark}

Let us now recall some facts about separable metric spaces.

Let \((X, d)\) be a metric space and let \(A \subseteq X\). A set \(S \subseteq A\) is said to be \emph{dense} in \(A\) if for every \(a \in A\) there is a sequence \((s_n)_{n\in \mathbb{N}}\) of points of \(S\) such that
\begin{equation*}
\lim_{n \to \infty} d(a, s_n) = 0.
\end{equation*}

\begin{definition}\label{d1.14}
A metric space \((X, d)\) is said to be \emph{separable} if \(X\) contains a dense countable subset.
\end{definition}

For the proof of the next lemma see, for example, Corollary~8.11 in~\cite{Kai2023ACTOMS}.

\begin{lemma}\label{l1.16}
Let \((X, d)\) be a separable metric space and let \(A\) be a nonempty subset of \(X\). Then the metric space \((A, d|_{A \times A})\) is also separable.
\end{lemma}

The following lemma is a part of Corollary~4.1.16 from~\cite{Eng1989}.

\begin{lemma}\label{l1.20}
Let \((X, d)\) be a separable metric space. Then every set of pairwise disjoint \(\bar{B} \in \bar{\BB}_X^{0}\) is countable.
\end{lemma}

\begin{definition}\label{d1.21}
Let \((Y, \delta)\) be a metric space and let \(S \subseteq Y\). A point \(s \in S\) is called an \emph{isolated} point of \(S\) if there exists \(r > 0\) such that \(S \cap \bar{B}_r(s) = \{s\}\). The collection of all isolated points of \(S\) in \((Y, \delta)\) will be denoted by \(\operatorname{iso}_Y(S)\). In the case \(S = Y\) we will also use the notation \(\operatorname{iso}(Y)\) instead of \(\operatorname{iso}_Y(Y)\).
\end{definition}

\begin{proposition}\label{p1.22}
Let \((X, d)\) be a separable metric space. Then \(\operatorname{iso}(X)\) is a countable set.
\end{proposition}

\begin{proof}
The set \(\operatorname{iso}(X)\) evidently is countable if this is the empty set. Suppose that \(\operatorname{iso}(X) \neq \varnothing\). Then, by Lemma~\ref{l1.16}, \(\operatorname{iso}(X)\) is a separable subspace of \((X, d)\). Now using the definition of the set \(\bar{\BB}_X^0\) and Definition~\ref{d1.21} we see that \(\{p\} \in \bar{\BB}_X^0\) holds for every \(p \in \operatorname{iso}(X)\). Hence, \(\operatorname{iso}(X)\) is countable by Lemma~\ref{l1.20}.
\end{proof}

\begin{remark}
New results related to separable ultrametric spaces can be found in~\cite{DR2025LTGSALFU, DS2022TRoUCaS, MSD2023CSOUSWA}.
\end{remark}

\begin{definition}\label{d1.22}
Let \((Y, \delta)\) be a metric space and let \(S \subseteq Y\). A point \(c \in Y\) is called an \emph{accumulation point} of \(S\) if, for every \(r > 0\), the closed ball \(\bar{B}_r(c)\) contains a point of \(S\) other than \(c\) itself. The set of all accumulation points of \(S\) in \((Y, \delta)\) will be denoted by \(\operatorname{acc}_Y(S)\). We will also use the notation \(\operatorname{acc}(Y)\) instead of \(\operatorname{acc}_Y(Y)\).
\end{definition}

It is easy to prove that the equality
\begin{equation}\label{s1:e5}
\operatorname{iso}_Y(S) \cap \operatorname{acc}_Y(S) = \varnothing
\end{equation}
holds for every \(S \subseteq Y\). Moreover, the equality
\begin{equation}\label{s1:e6}
\operatorname{iso}_Y(S) \cup \operatorname{acc}_Y(S) = Y
\end{equation}
is true if and only if \(S\) is a dense subset of \((Y, \delta)\).

The following result seems to be known, but the author cannot give a precise reference here.

\begin{proposition}\label{p1.12}
Let \((Y, \delta)\) be a metric space and let \(A\) and \(B\) be dense subsets of \(Y\). Then the equality
\begin{equation}\label{p1.12:e1}
\operatorname{iso}_Y(A) = \operatorname{iso}_Y(B)
\end{equation}
holds.
\end{proposition}

\begin{proof}
Equality~\eqref{p1.12:e1} holds if and only if we have
\begin{equation}\label{p1.12:e2}
\operatorname{iso}_Y(A) \subseteq \operatorname{iso}_Y(B)
\end{equation}
and
\begin{equation}\label{p1.12:e3}
\operatorname{iso}_Y(B) \subseteq \operatorname{iso}_Y(A).
\end{equation}
If~\eqref{p1.12:e2} is false, then there is a point \(p\) such that
\begin{equation}\label{p1.12:e4}
p \in \operatorname{iso}_Y(A)
\end{equation}
and
\begin{equation}\label{p1.12:e5}
p \notin \operatorname{iso}_Y(B).
\end{equation}
Since \(B\) is dense in \(Y\), \eqref{s1:e5}, \eqref{s1:e6} and \eqref{p1.12:e5} imply the equality
\begin{equation}\label{p1.12:e6}
\lim_{n \to \infty} \delta(p, b_n) = 0
\end{equation}
for some sequence \((b_n)_{n \in \mathbb{N}}\) of distinct points of \(B\). Since \(A\) is dense in \(Y\), there is a sequence \((a_n)_{n \in \mathbb{N}}\) of distinct points of \(A\) such that
\begin{equation}\label{p1.12:e7}
\lim_{n \to \infty} \delta(a_n, b_n) = 0.
\end{equation}
The triangle inequality and limit relations \eqref{p1.12:e6}, \eqref{p1.12:e7} give us
\[
\lim_{n \to \infty} \delta(p, a_n) \leqslant \lim_{n \to \infty} \delta(p, b_n) + \lim_{n \to \infty} \delta(a_n, b_n) = 0.
\]
Therefore, we have \(p \in \operatorname{acc}_Y(A)\). The last membership relation and \eqref{s1:e5} with \(S = A\) imply \(p \notin \operatorname{iso}_Y(A)\), contrary to \eqref{p1.12:e4}. Thus, inclusion~\eqref{p1.12:e2} is true. By similar reasoning, one can prove that~\eqref{p1.12:e3} is also true. Equality~\eqref{p1.12:e1} follows.
\end{proof}

To formulate the corollary of Proposition~\ref{p1.12}, we need the concept of a discrete subset of a metric space.

\begin{definition}\label{d1.23}
Let \((X, d)\) be a metric space and let \(A \subseteq X\). Then \(A\) is \emph{discrete} if every point of \(A\) is an isolated point of \(A\), i.e., if the equality
\[
A = \operatorname{iso}_X(A)
\]
holds. The metric space \((X, d)\) is discrete if the set \(X\) is discrete.
\end{definition}

\begin{corollary}\label{c1.13}
Let \((Y, \delta)\) be a metric space and let \(A\) and \(B\) be dense discrete subsets of \(Y\). Then the equality
\begin{equation}\label{c1.13:e1}
A = B
\end{equation}
holds.
\end{corollary}

\begin{proof}
Since \(A\) and \(B\) are discrete subsets of \(Y\), the equalities
\begin{equation}\label{c1.13:e2}
A = \operatorname{iso}_Y(A) \text{ and } B = \operatorname{iso}_Y(B)
\end{equation}
hold by Definition~\ref{d1.23}. Since \(A\) and \(B\) are dense subsets of \(Y\), Proposition~\ref{p1.12} gives us the equality
\begin{equation}\label{c1.13:e3}
\operatorname{iso}_Y(A) = \operatorname{iso}_Y(B).
\end{equation}
Equality~\eqref{c1.13:e1} follows from~\eqref{c1.13:e2} and \eqref{c1.13:e3}.
\end{proof}

Following~\cite[p.~48]{DD2009EOD}, we will say that a metric space \((X, d)\) is \emph{metrically discrete} if there exists \(t > 0\) such that \(d(x, y) \geqslant t\) for all distinct \(x\), \(y \in X\). Moreover, a metric \(d \colon X \times X \to \mathbb{R}^{+}\) is called \emph{equidistant} if there exists \(t > 0\) such that \(d(x, y) = t\) for all distinct \(x\), \(y \in X\) (see \cite[p.~46]{DD2009EOD}).

It is clear that every equidistant metric space is metrically discrete and that every metrically discrete metric space is discrete.

Another special class of discrete metric spaces is the class of locally finite metric spaces.

\begin{definition}\label{d1.26}
A metric space \((X, d)\) is called \emph{locally finite} if every bounded \(A \subseteq X\) is finite.
\end{definition}

\begin{remark}\label{r1.27}
New interconnections between locally finite ultrametric spaces and labeled trees were recently obtained in~\cite{DK2023LFUSALT, DR2025LTGSALFU}. Some results which are closely connected with equidistant metrics and their pseudometric generalization can be found in~\cite{DL2020CCOP, BD2024OMOMPF, BD2023WAPACS, BD2023PSFMTMITGO}.
\end{remark}

We conclude this section be recalling the concepts of locally compact metric spaces and boundedly compact ones.

\begin{definition}\label{d2.22}
A metric space \((X, d)\) is called \emph{locally compact} if, for each \(x \in X\), there is \(r > 0\) such that the closed ball \(\bar{B}_r(x)\) is a compact subset of \((X, d)\).
\end{definition}

A metric space \((X, d)\) is called \emph{boundedly compact} if every closed and bounded subset of \(X\) is compact (see, for example, \cite{BDC1991AGOBCMS}).

It is clear that every locally finite metric space is boundedly compact and each boundedly compact metric space is locally compact. Moreover, the classes of the discrete metric spaces and compact ones are subclasses of the class of locally compact metric spaces.

\section{Ultrametric balls}
\label{sec3}

The present section provides a brief overview of properties of closed ultrametric balls.

\begin{proposition}\label{p2.4}
Let \((X, d)\) be an ultrametric space. Then for each ball \(\bar{B}_r(c) \in \bar{\BB}_X\) and every \(a \in \bar{B}_r(c)\) we have
\[
\bar{B}_r(c) = \bar{B}_r(a).
\]
\end{proposition}

\begin{proof}
It directly follows from Proposition~18.4 of~\cite{Sch1985}.
\end{proof}

\begin{lemma}\label{l2.6}
Let \(\bar{B}_r(c)\) be a closed ball in an ultrametric space \((X, d)\) and let
\[
r_1 := \diam \bar{B}_r(c).
\]
Then the equality \(\bar{B}_{r_1}(c) = \bar{B}_r(c)\) holds.
\end{lemma}

\begin{proof}
Using the definition of closed balls and equality \eqref{e2.3} with \(A = \bar{B}_r(c)\) and \(p_0 = c\), we obtain the inequality \(r_1 \leqslant r\). Consequently, the inclusion \(\bar{B}_{r_1}(c) \subseteq \bar{B}_r(c)\) holds. To prove the converse inclusion it suffices to note that
\[
d(x, c) \leqslant \diam \bar{B}_r(c) = r_1
\]
holds for every \(x \in \bar{B}_r(c)\), that implies \(\bar{B}_r(c) \subseteq \bar{B}_{r_1}(c)\).
\end{proof}

\begin{proposition}\label{p2.7}
Let \((X, d)\) be an ultrametric space and let \(\bar{B}_1\), \(\bar{B}_2 \in \bar{\mathbf{B}}_X\). Then $\bar{B}_1 \subseteq \bar{B}_2$ holds if and only if $\bar{B}_1 \cap \bar{B}_2 \neq \varnothing$ and \(\diam \bar{B}_1 \leqslant \diam \bar{B}_2\).
\end{proposition}

\begin{proof}
It follows from Proposition~\ref{p2.4} and Lemma~\ref{l2.6}.
\end{proof}

Proposition~\ref{p2.7} implies the next corollary.

\begin{corollary}\label{c3.4}
Let \((X, d)\) and \((Y, \rho)\) be ultrametric spaces such that \(Y \in \bar{\BB}_X\) and \(\rho = d|_{Y \times Y}\). Then the equality
\[
\bar{\BB}_Y = \{\bar{B} \colon \bar{B} \in \bar{\BB}_X \text{ and } \bar{B} \subseteq Y\}
\]
holds.
\end{corollary}

\begin{example}\label{ex2.6}
Following Delhomm\'{e}, Laflamme, Pouzet and Sauer \cite{DLPS2008TaiA}, we define an ultrametric \(d^{+} \colon \RR^{+} \times \RR^{+} \to \RR^{+}\) as
\begin{equation}\label{ex2.6:e1}
d^{+}(x, y) := \begin{cases}
0 & \text{if } x = y,\\
\max\{x, y\} & \text{if } x \neq y.
\end{cases}
\end{equation}
Let \(c > 0\) and let \(\bar{B}_r(c)\) be a closed ball in the space \((\RR^{+}, d^{+})\). Then the equality
\[
\bar{B}_r(c) = \{c\}
\]
holds if and only if \(0 \leqslant r < c\). Furthermore, we have the equality
\[
\bar{B}_r(c) = [0, r]
\]
for every \(r \geqslant c\).
\end{example}

\begin{remark}
The interesting applications of the Delhomm\'{e}---Laflamme---Pouzet---Sauer space \((\RR^+, d^+)\) are given in~\cite{Ish2023COUUUS}. Properties of endomorphisms of the groupoid \((\RR^+, d^+)\) have been described in \cite{DK2024DLPSSAG}.
\end{remark}

The next proposition is a modification of Proposition~18.5 of \cite{Sch1985}.

\begin{proposition}\label{p2.5}
Let \((X, d)\) be an ultrametric space and let \(\bar{B}_1\), \(\bar{B}_2 \in \bar{\BB}_X\) be different. If \(\bar{B}_1 \cap \bar{B}_2 \neq \varnothing\) holds, then we have either \(\bar{B}_1 \subset \bar{B}_2\) or \(\bar{B}_2 \subset \bar{B}_1\). If \(\bar{B}_1\) and \(\bar{B}_2\) are disjoint, then the equality
\begin{equation}\label{p2.5:e1}
d(x_1, x_2) = \diam (\bar{B}_1 \cup \bar{B}_2)
\end{equation}
holds for every \(x_1 \in \bar{B}_1\) and every \(x_2 \in \bar{B}_2\).
\end{proposition}

\begin{proof}
By Proposition~18.5 of \cite{Sch1985}, we have
\[
(\bar{B}_1 \subset \bar{B}_2 \text{ or } \bar{B}_2 \subset \bar{B}_1) \Leftrightarrow (\bar{B}_1 \cap \bar{B}_2 \neq \varnothing)
\]
and, moreover, the same proposition claims the equality
\begin{equation}\label{p2.5:e5}
d(x_1, x_2) = d(y_1, y_2)
\end{equation}
for all \(x_1\), \(y_1 \in \bar{B}_1\) and all \(x_2\), \(y_2 \in \bar{B}_2\) whenever \(\bar{B}_1 \cap \bar{B}_2 = \varnothing\).

Write
\begin{equation}\label{p2.5:e6}
\Delta(\bar{B}_1, \bar{B}_2) := \sup\{d(x_1, x_2) \colon x_1 \in \bar{B}_1 \text{ and } x_2 \in \bar{B}_2\}.
\end{equation}
It follows directly from~\eqref{p2.5:e6} that
\begin{equation}\label{p2.5:e2}
\diam (\bar{B}_1 \cup \bar{B}_2) = \max \bigl\{\diam \bar{B}_1, \diam \bar{B}_2, \Delta(\bar{B}_1, \bar{B}_2)\bigr\}.
\end{equation}
Now, using \eqref{p2.5:e5} and \eqref{p2.5:e2} we see that \eqref{p2.5:e1} holds if
\begin{equation}\label{p2.5:e3}
\diam \bar{B}_i \leqslant \Delta(\bar{B}_1, \bar{B}_2), \quad i = 1, 2.
\end{equation}
Without loss of generality, we may assume \(\diam \bar{B}_1 \geqslant \diam \bar{B}_2\). Now, if \eqref{p2.5:e3} does not hold, then there are \(z_1\), \(y_1 \in \bar{B}_1\) such that
\begin{equation}\label{p2.5:e4}
d(z_1, y_1) > \Delta(\bar{B}_1, \bar{B}_2).
\end{equation}
Let \(r_1\) be the radius of \(\bar{B}_1\). Using Proposition~\ref{p2.4} we can consider \(z_1\) as a center of the ball \(\bar{B}_1\). By inequality~\eqref{p2.5:e4},
\[
r_1 \geqslant d(z_1, y_1) > d(z_1, x_2)
\]
holds for every \(x_2 \in \bar{B}_2\). Thus, \(\bar{B}_2 \subseteq \bar{B}_1\), contrary to \(\bar{B}_1 \cap \bar{B}_2 = \varnothing\).
\end{proof}

\begin{remark}\label{r2.13}
A partial generalization of Propositions~\ref{p2.4}, \ref{p2.7} and \ref{p2.5}, for ``balls'' whose points are finite, nonempty subsets of \(X\), is given in Proposition~2.1~\cite{DDP2011pNUAA}.
\end{remark}

Following~\cite{Dov2019pNUAA} we introduce now the concept of smallest ball containing a given bounded subset of ultrametric space.

\begin{definition}\label{d2.9}
Let \((X, d)\) be a metric space and let \(A\) be a nonempty, bounded subset of \(X\). A ball \(B^* = B^*(A) \in \bar{\mathbf{B}}_X\) is the \emph{smallest ball} containing \(A\) if we have \(B^* \supseteq A\) and the implication
\begin{equation}\label{d2.9:e1}
(\bar{B} \supseteq A) \Rightarrow (\bar{B} \supseteq B^*)
\end{equation}
is valid for every \(\bar{B} \in \bar{\mathbf{B}}_X\).
\end{definition}

It is clear that, for every nonempty, bounded \(A \subseteq X\), the smallest ball \(B^*(A)\) is unique if it exists.

The next proposition follows from Proposition~\ref{p2.7} (see~\cite{Dov2019pNUAA}, Proposition~1.8, for details).

\begin{proposition}\label{p2.12}
Let \((X, d)\) be an ultrametric space, let \(A\) be a bounded subset of \(X\), \(r := \diam A\), and let \(c \in A\). Then the closed ball \(\bar{B}_r(c)\) is the smallest ball containing \(A\).
\end{proposition}

\section{Dense discrete subsets of ultrametric balleans}
\label{sec4}

Let \(X\) be a metric space. In what follows, we will use the notation \(\mathfrak{M}_X\) for the set of all nonempty closed bounded subsets of \(X\). If \(d \colon X \times X \to \RR^{+}\) is a metric on \(X\), then the \emph{Hausdorff distance} \(d_H(A, B)\) between \(A\), \(B \in \mathfrak{M}_X\) is defined by
\begin{equation}\label{e6.11}
d_H(A, B) := \max\left\{\sup_{a \in A} \inf_{b \in B} d(a, b), \sup_{b \in B} \inf_{a \in A} d(a, b)\right\}.
\end{equation}

It is well-known that the Hausdorff distance \(d_H\) is a metric on \(\mathfrak{M}_X\) for arbitrary metric space \((X, d)\) (see, for example, Theorem~7.3.1 of~\cite{Sea2007}).

\begin{proposition}\label{p3.2}
Let \((X, d)\) be a metric space. Then \(\bar{\BB}_X\) is a nonempty subset of \(\mathfrak{M}_X\).
\end{proposition}

\begin{proof}
Each \(\bar{B} \in \bar{\BB}_X\) is a closed bounded subset of \(X\). Since the set \(X\) is nonempty by definition of metric space, the set \(\bar{\BB}_X\) is also nonempty by definition of the closed balls. Hence, \(\bar{\BB}_X\) is a nonempty subset of \(\mathfrak{M}_X\), as required.
\end{proof}

The following lemma is a part of Lemma~2.1 from~\cite{Qiu2014pNUAA}.

\begin{lemma}\label{l3.3}
Let \((X, d)\) be an ultrametric space and let \(\bar{B}_1\), \(\bar{B}_2 \in \bar{\BB}_X^0\) be different. Then the equality
\begin{equation}\label{l3.3:e1}
d_H(\bar{B}_1, \bar{B}_2) = \begin{cases}
\dist(\bar{B}_1, \bar{B}_2) & \text{if \(\bar{B}_1 \cap \bar{B}_2 = \varnothing\)},\\
\max\{\diam \bar{B}_1, \diam \bar{B}_2\} & \text{if \(\bar{B}_1 \cap \bar{B}_2 \neq \varnothing\)}
\end{cases}
\end{equation}
holds, where
\[
\dist(\bar{B}_1, \bar{B}_2) = \inf\{d(x, y) \colon x \in \bar{B}_1, y \in \bar{B}_2\}.
\]
\end{lemma}

Equality~\eqref{l3.3:e1} admits the following modification, which was obtained earlier in \cite{Dov2019pNUAA} for finite ultrametric spaces.

\begin{lemma}\label{l6.12}
Let \((X, d)\) be an ultrametric space. Then the equality
\begin{equation}\label{l6.12:e1}
d_H(\bar{B}_1, \bar{B}_2) = \diam(\bar{B}_1 \cup \bar{B}_2)
\end{equation}
holds for all distinct \(\bar{B}_1\), \(\bar{B}_2 \in \bar{\mathbf{B}}_X\).
\end{lemma}

\begin{proof}
Let \(\bar{B}_1\), \(\bar{B}_2 \in \bar{\mathbf{B}}_X\) be distinct.

If \(\bar{B}_1 \cap \bar{B}_2 = \varnothing\) holds, then equality \eqref{l6.12:e1} follows from \eqref{p2.5:e1} and \eqref{e6.11}.

Let us consider the case \(\bar{B}_1 \cap \bar{B}_2 \neq \varnothing\). Then, by Proposition~\ref{p2.5}, we have either \(\bar{B}_1 \subset \bar{B}_2\) or \(\bar{B}_2 \subset \bar{B}_1\). Without loss of generality we assume
\begin{equation}\label{l6.12:e2}
\bar{B}_1 \subset \bar{B}_2.
\end{equation}
If \(|\bar{B}_1| = 1\), then there is \(p_1 \in X\) such that
\begin{equation}\label{l6.12:e3}
\bar{B}_1  = \{p_1\}.
\end{equation}
Inclusion~\eqref{l6.12:e2} implies that \(p_1 \in \bar{B}_2\). Using \eqref{e2.3}, \eqref{e6.11} and \eqref{l6.12:e3} we obtain
\begin{align*}
d_H(\bar{B}_1, \bar{B}_2) & = \max\left\{\inf_{p_2 \in \bar{B}_2} d(p_1, p_2), \sup_{p_2 \in \bar{B}_2} d(p_1, p_2)\right\} \\
& = \sup\{d(p_1, p_2) \colon p_2 \in \bar{B}_2\} = \diam \bar{B}_2.
\end{align*}
From \eqref{l6.12:e2} it follows that \(\diam \bar{B}_2 = \diam (\bar{B}_1 \cup \bar{B}_2)\). Thus, \eqref{l6.12:e1} holds.

To complete the proof we must consider the case when
\begin{equation}\label{l6.12:e4}
\bar{B}_1 \subset \bar{B}_2
\end{equation}
and \(\diam \bar{B}_1 > 0\). If \eqref{l6.12:e4} is true, then, in accordance with Lemma~\ref{l3.3},
\begin{equation}\label{l6.12:e5}
d_H(\bar{B}_1, \bar{B}_2) = \max\{\diam \bar{B}_1, \diam \bar{B}_2\}
\end{equation}
holds. Inclusion~\eqref{l6.12:e4} implies the equalities
\begin{equation}\label{l6.12:e6}
\bar{B}_1 \cup \bar{B}_2 = \bar{B}_2
\end{equation}
and
\begin{equation}\label{l6.12:e7}
\max\{\diam \bar{B}_1, \diam \bar{B}_2\} = \diam \bar{B}_2.
\end{equation}
Now~\eqref{l6.12:e1} follows from \eqref{l6.12:e5}, \eqref{l6.12:e6} and \eqref{l6.12:e7}.
\end{proof}

Proposition~\ref{p2.12} and Lemma~\ref{l6.12} give us the following geometric interpretation of the Hausdorff distance between closed ultrametric balls.

\begin{proposition}\label{c3.5}
Let \((X, d)\) be an ultrametric space and let \(\bar{B}_1\), \(\bar{B}_2 \in \bar{\BB}_X\) be distinct. Then the equality
\begin{equation}\label{c3.5:e1}
d_H(\bar{B}_1, \bar{B}_2) = \diam B^*
\end{equation}
holds, where \(B^* = B^*(\bar{B}_1 \cup \bar{B}_2)\) is the smallest closed ball containing the union \(\bar{B}_1 \cup \bar{B}_2\).
\end{proposition}

\begin{proof}
By Lemma~\ref{l6.12}, we have the equality
\[
d_H(\bar{B}_1, \bar{B}_2) = \diam (\bar{B}_1 \cup \bar{B}_2).
\]
Hence, \eqref{c3.5:e1} holds iff
\begin{equation}\label{c3.5:e2}
\diam (\bar{B}_1 \cup \bar{B}_2) = \diam B^*.
\end{equation}
The inclusion \(\bar{B}_1 \cup \bar{B}_2 \subseteq B^*\) implies the inequality
\begin{equation}\label{c3.5:e3}
\diam (\bar{B}_1 \cup \bar{B}_2) \leqslant \diam B^*.
\end{equation}
Moreover, using Proposition~\ref{p2.12}, we obtain the equality
\begin{equation}\label{c3.5:e4}
B^* = \bar{B}_r(c^*),
\end{equation}
where \(c^*\) is an arbitrary point of \(B^*\) and
\begin{equation}\label{c3.5:e5}
r = \diam (\bar{B}_1 \cup \bar{B}_2).
\end{equation}
Consequently,  \eqref{c3.5:e4}, \eqref{c3.5:e5} and equality~\eqref{e2.3} with \(A = B^*\) and \(p_0 = c^*\) give us
\begin{equation}\label{c3.5:e6}
\diam B^* \leqslant r = \diam (\bar{B}_1 \cup \bar{B}_2).
\end{equation}
Now~\eqref{c3.5:e2} follows from \eqref{c3.5:e3} and \eqref{c3.5:e6}.
\end{proof}

\begin{theorem}\label{t3.2}
Let \((X, d)\) be a metric space. Then the set
\begin{equation}\label{t3.2:e1}
\hat{X} := \bigl\{\{x\} \colon x\in X\bigr\}
\end{equation}
is an isometric copy of \((X, d)\) in \((\mathfrak{M}_X, d_H)\) and it is a closed subset of \(\mathfrak{M}_X\).
\end{theorem}

For the proof see Theorem~10.7.1 of~\cite{Sea2007}.

Since for every \(x \in X\) we have \(\{x\} \in \bar{\BB}_X\), Theorem~\ref{t3.2} and Proposition~\ref{p3.2} give us the following.

\begin{corollary}\label{c3.7}
Let \((X, d)\) be a metric space. Then \((\hat{X}, d_H)\) is an isometric copy of \((X, d)\) in \((\bar{\BB}_X, d_H)\) and \(\hat{X}\) is a closed subset of \(\bar{\BB}_X\).
\end{corollary}

\begin{remark}\label{r3.8}
In Corollary~\ref{c3.7} we use the symbol \(d_H\) to denote the restrictions of the Hausdorff metric
\[
d_H \colon \mathfrak{M}_X \times \mathfrak{M}_X \to \mathbb{R}^+
\]
on the sets \(\bar{\BB}_X \times \bar{\BB}_X\) and \(\hat{X} \times \hat{X}\). For simplicity, we will use the same notation \(d_H\) for the metric \(d_H|_{\mathbf{A} \times \mathbf{A}}\) in the case of any nonempty \(\mathbf{A} \subseteq \mathfrak{M}_X\).
\end{remark}

The following result can be derived from Theorem~\ref{t3.2} and Lemma~2.4 of \cite{Qiu2014pNUAA}.

\begin{proposition}\label{p3.4}
Let \((X, d)\) be a metric space. Then \((\mathfrak{M}_X, d_H)\) is ultrametric if and only if \((X, d)\) is ultrametric.
\end{proposition}

The next theorem can be considered as a ``ballean'' modification of Proposition~\ref{p3.4}.

\begin{theorem}\label{l6.13}
Let \((X, d)\) be a metric space. Then the following statements are equivalent:
\begin{enumerate}
\item\label{l6.13:s1} The metric space \((X, d)\) is ultrametric.
\item\label{l6.13:s2} The metric space \((\bar{\BB}_X, d_H)\) is ultrametric.
\end{enumerate}
\end{theorem}

\begin{proof}
\(\ref{l6.13:s1} \Rightarrow \ref{l6.13:s2}\). Let \((X, d)\) be an ultrametric space. Then \((\mathfrak{M}_X, d_H)\) is ultrametric by Proposition~\ref{p3.4}. Hence, \((\bar{\BB}_X, d_H)\) is ultrametric as a subspace of the ultrametric space \((\mathfrak{M}_X, d_H)\).

\(\ref{l6.13:s2} \Rightarrow \ref{l6.13:s1}\). Let \((\bar{\BB}_X, d_H)\) be an ultrametric space and let
\[
\hat{X} := \bigl\{\{x\} \colon x\in X\bigr\}.
\]
By Corollary~\ref{c3.7}, \((\hat{X}, d_H)\) and \((X, d)\) are isometric and \(\hat{X} \subseteq \bar{\BB}_X\) holds. Consequently, \((\hat{X}, d_H)\) is ultrametric as a subspace of \((\bar{\BB}_X, d_H)\) and, in addition, \((X, d)\) is ultrametric as a space isometric to \((\hat{X}, d_H)\).
\end{proof}

The next proposition gives us a generalization of Proposition~\ref{c3.5}.

\begin{proposition}\label{p3.7}
Let \((X, d)\) be an ultrametric space, \(\mathbf{A}\) be a bounded subset of \((\bar{\BB}_X, d_H)\), and let
\begin{equation}\label{p3.7:e1}
A := \bigcup_{\bar{B} \in \mathbf{A}} \bar{B}.
\end{equation}
If \(\mathbf{A}\) contains at least two elements, then the equalities
\begin{equation}\label{p3.7:e2}
\diam \mathbf{A} = \diam A = \diam B^*
\end{equation}
hold, where \(\diam \mathbf{A}\) is the diameter of the set \(\mathbf{A}\) in \((\bar{\BB}_X, d_H)\), \(\diam A\) is the the diameter of the set \(A\) in \((X, d)\) and \(B^* = B^*(A)\) is the smallest ball containing the set \(A\).
\end{proposition}

\begin{proof}
Suppose that \(\mathbf{A}\) has at least two elements and consider an arbitrary \(\bar{B}_0 \in \mathbf{A}\). The metric space \((\bar{\BB}_X, d_H)\) is ultrametric by Theorem~\ref{l6.13}. Hence, using formula~\eqref{e2.3} with \((X, d) = (\bar{\BB}_X, d_H)\) and \(A = \mathbf{A}\) and \(p_0 = \bar{B}_0\), we obtain the equality
\[
\diam \mathbf{A} = \sup\{d_H(\bar{B}_0, \bar{B}) \colon \bar{B} \in \mathbf{A}\}.
\]
The last equality and \eqref{l6.12:e1} imply the equality
\begin{equation}\label{p3.7:e3}
\diam \mathbf{A} = \sup_{\bar{B} \in \mathbf{A}} \diam (\bar{B}_0 \cup \bar{B}).
\end{equation}
Let \(b_0\) be a point of \(\bar{B}_0\). Then formula~\eqref{e2.3} with \(A = \bar{B}_0 \cup \bar{B}\) and \(p_0 = b_0\) gives us
\begin{equation}\label{p3.7:e4}
\diam (\bar{B}_0 \cup \bar{B}) = \sup \{d(x, b_0) \colon x \in \bar{B}_0 \cup \bar{B}\}.
\end{equation}
Equality~\eqref{p3.7:e1} implies the inclusion \(\bar{B} \subseteq A\) for every \(\bar{B} \in \mathbf{A}\). Hence, using~\eqref{p3.7:e3} and \eqref{p3.7:e4}, we obtain
\[
\diam \mathbf{A} \leqslant \sup \{d(x, b_0) \colon x \in A\}.
\]
The last inequality and equality~\eqref{e2.3} with \(p_0 = b_0\) imply
\begin{equation}\label{p3.7:e5}
\diam \mathbf{A} \leqslant \sup \{d(x, b_0) \colon x \in A\} = \diam A.
\end{equation}
Thus, the inequality
\[
\diam \mathbf{A} \leqslant \diam A
\]
holds. To prove the converse inequality
\begin{equation}\label{p3.7:e6}
\diam \mathbf{A} \geqslant \diam A,
\end{equation}
it suffices to note that, for every \(x \in A\), there is \(\bar{B} \in \mathbf{A}\) such that
\[
\{x, b_0\} \subseteq \bar{B}_0 \cup \bar{B}.
\]
The last inclusion and \eqref{p3.7:e3} give the inequality
\[
d(x, b_0) \leqslant \diam \mathbf{A}
\]
for every \(x \in A\), that implies \eqref{p3.7:e6} by using~\eqref{e2.3}. The first equality in double equality~\eqref{p3.7:e2} follows.

To prove the second one we note that \(A\) is bounded in \((X, d)\) because \(\mathbf{A}\) is bounded in \((\bar{\BB}_X, d_H)\) and
\[
\diam \mathbf{A} = \diam A
\]
holds (as was proved above). Hence, \(\diam A = \diam B^*\) holds by Proposition~\ref{p2.12}.
\end{proof}

\begin{remark}\label{r3.9}
If the set \(\mathbf{A}\) contains a single element \(\bar{B}\), then equality~\eqref{p3.7:e2} is satisfied if and only if \(\diam \bar{B} = 0\) holds, i.e., iff there is \(c \in X\) such that \(\bar{B} = \{c\}\). In this case we have \(\mathbf{A} = \{\{c\}\}\), where \(\{\{c\}\}\) is the set containing the single element \(\{c\}\).
\end{remark}

The next theorem gives us an interconnection between isolated points of ultrametric spaces and isolated points of their balleans.

\begin{theorem}\label{t3.11}
Let \((X, d)\) be an ultrametric space. Then the equalities
\begin{equation}\label{t3.11:e1}
\bar{\BB}_X^0 = \{\bar{B} \in \bar{\BB}_X \colon \diam \bar{B} > 0\} \cup \bigl\{\{x\} \colon x \in \operatorname{iso}(X)\bigr\} = \operatorname{iso}(\bar{\BB}_X)
\end{equation}
hold.
\end{theorem}

\begin{proof}
First we will prove the inclusions
\begin{gather}
\label{t3.11:e2.1}
\bar{\BB}_X^0 \subseteq \{\bar{B} \in \bar{\BB}_X \colon \diam \bar{B} > 0\} \cup \bigl\{\{x\} \colon x \in \operatorname{iso}(X)\bigr\},\\
\label{t3.11:e2.2}
\{\bar{B} \in \bar{\BB}_X \colon \diam \bar{B} > 0\} \cup \bigl\{\{x\} \colon x \in \operatorname{iso}(X)\bigr\} \subseteq \operatorname{iso} (\bar{\BB}_X)
\end{gather}
and
\begin{equation}\label{t3.11:e3}
\bar{\BB}_X \setminus \bar{\BB}_X^0 \subseteq \operatorname{acc}(\bar{\BB}_X).
\end{equation}

Let us prove inclusion~\eqref{t3.11:e2.1}. Let us consider an arbitrary \(\bar{B}_0 \in \bar{\BB}_X^0\). To prove~\eqref{t3.11:e2.1}, it is enough to show that
\begin{equation}\label{t3.11:e3.1}
\bar{B}_0 \in \{\bar{B} \in \bar{\BB}_X \colon \diam \bar{B} > 0\} \cup \bigl\{\{x\} \colon x \in \operatorname{iso}(X)\bigr\}.
\end{equation}
If \(\diam \bar{B}_0\) is strictly positive, then~\eqref{t3.11:e3.1} evidently holds. Let us consider the case \(\diam \bar{B}_0 = 0\). The last equality implies the existence of a point \(x_0 \in X\) such that
\begin{equation}\label{t3.11:e3.7}
\bar{B}_0 = \{x_0\}.
\end{equation}
It follows from the definition of the set \(\bar{\BB}_X^0\) that the equality
\begin{equation}\label{t3.11:e3.8}
\bar{B}_0 = \bar{B}_{r_0}(x_0)
\end{equation}
holds for some \(r_0 > 0\). Equalities~\eqref{t3.11:e3.7} and \eqref{t3.11:e3.8} imply that \(x_0\) is an isolated point of \((X, d)\) by Definition~\ref{d1.21}. Thus, we have
\[
\bar{B}_0 = \{x_0\} \in \bigl\{\{x\} \colon x \in \operatorname{iso}(X)\bigr\}.
\]
Membership~\eqref{t3.11:e3.1} follows.

Let us prove inclusion~\eqref{t3.11:e2.2}. To do this, we consider an arbitrary
\[
\bar{B}_0 \in \{\bar{B} \in \bar{\BB}_X \colon \diam \bar{B} > 0\} \cup \bigl\{\{x\} \colon x \in \operatorname{iso}(X)\bigr\}
\]
and show that
\begin{equation}\label{t3.11:e4}
\bar{B}_0 \in \operatorname{iso} (\bar{\BB}_X).
\end{equation}
Suppose first that \(\diam \bar{B}_0 = r_0 > 0\). Then, for every \(\bar{B} \in \bar{\BB}_X\) different from \(\bar{B}_0\), equality~\eqref{l6.12:e1} with \(B_1 = \bar{B}_0\) and \(B_2 = \bar{B}\) gives us
\[
d_H(\bar{B}_0, \bar{B}) \geqslant r_0 > 0,
\]
that implies~\eqref{t3.11:e4}.

Let us consider the case when \(\bar{B}_0 \in \bigl\{\{x\} \colon x \in \operatorname{iso}(X)\bigr\}\). Let \(c_0\) be an isolated point of \((X, d)\) such that
\begin{equation}\label{t3.11:e5}
\bar{B}_0 = \{c_0\}
\end{equation}
The membership \(c_0 \in \operatorname{iso}(X)\) implies that there is \(r_0 > 0\) such that
\[
d(c_0, x) > r_0
\]
for every \(x \in X \setminus \{c_0\}\). If \(\bar{B} \in \bar{\BB}_X\) and \(\bar{B} \neq \bar{B}_0\) hold, then \(\bar{B}\) contains a point \(x \in X\) such that \(x \neq c_0\). Consequently, we have the inclusion
\[
\{x, c_0\} \subseteq \bar{B} \cup \bar{B}_0.
\]
The last inclusion, formula~\eqref{l6.12:e1} with \(B_1 = \bar{B}_0\) and \(B_2 = \bar{B}\) and equality~\eqref{t3.11:e5} imply
\[
d_H(\bar{B}_0, \bar{B}) \geqslant \diam (\bar{B} \cup \bar{B}_0) \geqslant \diam \{x, c_0\} > r_0.
\]

Thus, \eqref{t3.11:e4} holds for all \(\bar{B} \in \bar{\BB}_X\) different from \(\bar{B}_0\). Inclusion~\eqref{t3.11:e2.2} follows.

Let us prove~\eqref{t3.11:e3}. Let us consider an arbitrary
\begin{equation}\label{t3.11:e6}
\bar{B}_1 \in \bar{\BB}_X \setminus \bar{\BB}_X^0.
\end{equation}
If \(\diam \bar{B}_1 > 0\) holds, the, using Proposition~\ref{p2.4}, we obtain the membership \(\bar{B}_1 \in \bar{\BB}_X^0\) contrary to \eqref{t3.11:e6}. Hence, the equality \(\diam \bar{B}_1 = 0\) holds. The last equality implies the existence of a point \(c_1 \in X\) such that \(\bar{B}_1 = \{c_1\}\). If \(c_1\) is an isolated point in the space \((X, d)\), then there exists \(r_1 > 0\) such that \(\bar{B}_1 = \bar{B}_{r_1}(c_1)\). Therefore, \(\bar{B}_1\) belongs to \(\bar{\BB}_X^0\) despite \eqref{t3.11:e6}. Consequently, \(\bar{B}_1 = \{c_1\}\) holds for some \(c_1 \in \operatorname{acc}(X)\). Hence, there is a sequence \((x_n)_{n \in \mathbb{N}}\) of distinct points of \(X\) such that
\begin{equation}\label{t3.11:e7}
\lim_{n \to \infty} d(c_1, x_n) = 0.
\end{equation}
By Corollary~\ref{c3.7}, the set
\[
\hat{X} := \bigl\{\{x\} \colon x\in X\bigr\}
\]
is an isometric copy of \((X, d)\) in \((\bar{\BB}_X, d_H)\). Hence, \eqref{t3.11:e7} gives us
\[
\lim_{n \to \infty} d_H(\{c_1\}, \{x_n\}) = 0.
\]
So we have \(\bar{B}_1 \in \operatorname{acc}(\bar{\BB}_X)\). Inclusion~\eqref{t3.11:e3} follows.

Let us prove equalities~\eqref{t3.11:e1}. It follows from~\eqref{t3.11:e2.1} and \eqref{t3.11:e2.2} that \eqref{t3.11:e1} holds if and only if
\begin{equation}\label{t3.11:e7.1}
\bar{\BB}_X^0 = \operatorname{iso} (\bar{\BB}_X).
\end{equation}
Moreover, inclusions~\eqref{t3.11:e2.1} and \eqref{t3.11:e2.2} gives us the inclusion
\[
\bar{\BB}_X^0 \subseteq \operatorname{iso} (\bar{\BB}_X).
\]
The last inclusion implies that \eqref{t3.11:e7.1} is false if and only if there is
\begin{equation}\label{t3.11:e8}
\bar{B}_0 \in \operatorname{iso} (\bar{\BB}_X) \setminus \bar{\BB}_X^0.
\end{equation}
It follows from~\eqref{t3.11:e8} that \(\bar{B}_0 \in \bar{\BB}_X \setminus \bar{\BB}_X^0\). Hence, we have
\begin{equation}\label{t3.11:e9}
\bar{B}_0 \in \operatorname{acc}(\bar{\BB}_X)
\end{equation}
by~\eqref{t3.11:e3}. Now using~\eqref{t3.11:e8} and \eqref{t3.11:e9} we obtain
\[
\bar{B}_0 \in \operatorname{iso} (\bar{\BB}_X) \cap \operatorname{acc}(\bar{\BB}_X),
\]
that contradicts
\[
\operatorname{iso} (\bar{\BB}_X) \cap \operatorname{acc}(\bar{\BB}_X) = \varnothing.
\]
Thus, equality~\eqref{t3.11:e7.1} is fulfilled. Equality \eqref{t3.11:e1} follows.
\end{proof}

\begin{corollary}\label{c3.18}
The equality
\begin{equation}\label{c3.18:e1}
\operatorname{acc}(\bar{\BB}_X) = \bigl\{\{x\} \colon x \in \operatorname{acc}(X)\bigr\}
\end{equation}
holds for every ultrametric space \((X, d)\).
\end{corollary}

\begin{proof}
Equalities~\eqref{s1:e5} and \eqref{s1:e6} with \((Y, \delta) = (\bar{\BB}_X, d_H)\) and \(S = \bar{\BB}_X\) give us
\begin{equation}\label{c3.18:e2}
\bar{\BB}_X = \operatorname{acc}(\bar{\BB}_X) \cup \operatorname{iso}(\bar{\BB}_X)
\end{equation}
and
\begin{equation}\label{c3.18:e3}
\operatorname{acc}(\bar{\BB}_X) \cap \operatorname{iso}(\bar{\BB}_X) = \varnothing.
\end{equation}
By~\eqref{t3.11:e1} we have
\begin{equation}\label{c3.18:e4}
\operatorname{iso}(\bar{\BB}_X) = \{\bar{B} \in \bar{\BB}_X \colon \diam \bar{B} > 0\} \cup \bigl\{\{x\} \colon x \in \operatorname{iso}(X)\bigr\}.
\end{equation}
Moreover, it follows from the definition of the set \(\bar{\BB}_X\) that
\begin{equation}\label{c3.18:e5}
\bar{\BB}_X = \{\bar{B} \in \bar{\BB}_X \colon \diam \bar{B} > 0\} \cup \hat{X},
\end{equation}
where
\[
\hat{X} = \bigl\{\{x\} \colon x \in X\bigr\}.
\]
Now using the equality
\[
\hat{X} = \bigl\{\{x\} \colon x \in \operatorname{acc}(X)\bigr\} \cup \bigl\{\{x\} \colon x \in \operatorname{iso}(X)\bigr\},
\]
we can rewrite~\eqref{c3.18:e5} as
\begin{equation}\label{c3.18:e6}
\bar{\BB}_X = \{\bar{B} \in \bar{\BB}_X \colon \diam \bar{B} > 0\} \cup \bigl\{\{x\} \colon x \in \operatorname{acc}(X)\bigr\} \cup \bigl\{\{x\} \colon x \in \operatorname{iso}(X)\bigr\}.
\end{equation}
Since all sets in union~\eqref{c3.18:e6} are disjoint, equality~\eqref{c3.18:e1} follows from~\eqref{c3.18:e2}---\eqref{c3.18:e4}.
\end{proof}

\begin{proposition}\label{p3.21}
Let \((X, d)\) be an ultrametric space. Then the set \(\bar{\BB}_X^0\) is an unique dense discrete subset of the ultrametric space \((\bar{\BB}_X, d_H)\).
\end{proposition}

\begin{proof}
By Theorem~\ref{t3.11}, we have the equality
\begin{equation}\label{t3.14:e0}
\bar{\BB}_X^0 = \operatorname{iso}(\bar{\BB}_X).
\end{equation}
Hence, every \(\bar{B} \in \bar{\BB}_X^0\) is an isolated point of \((\bar{\BB}_X, d_H)\) and, consequently, it is an isolated point of \((\bar{\BB}_X^0, d_H)\). Thus, \(\bar{\BB}_X^0\) is a discrete subset of \((\bar{\BB}_X, d_H)\) by Definition~\ref{d1.23}.

Let us prove that \(\bar{\BB}_X^0\) is a dense subset of \((\bar{\BB}_X, d_H)\).

Let us consider an arbitrary \(\bar{B} = \bar{B}_r(c) \in \bar{\BB}_X\) and let \(r_1 := \diam \bar{B}\). Then, by Lemma~\ref{l2.6}, we have the equality \(\bar{B} = \bar{B}_{r_1}(c)\). Suppose first that the inequality \(r_1 > 0\) holds. Then, by definition of \(\bar{\BB}_X^0\), the membership relation \(\bar{B} \in \bar{\BB}_X^0\) is true. Consequently, the constant sequence \((\bar{B}_n)_{n \in \mathbb{N}}\) with \(\bar{B}_n = \bar{B}\) for all \(n \in \mathbb{N}\) converges to \(\bar{B}\) in \((\bar{\BB}_X, d_H)\).

If the equality \(r_1 = 0\) holds, then the set \(\bar{B}\) is a singleton
\begin{equation}\label{t3.14:e1}
\bar{B} = \{c\},
\end{equation}
where \(c\) is a point of \(X\). We claim that the sequence \((\bar{B}_m)_{m \in \mathbb{N}}\), with
\begin{equation}\label{t3.14:e2}
\bar{B}_m := \bar{B}_{1/m}(c)
\end{equation}
for every \(m \in \mathbb{N}\), converges to \(\bar{B}\) in \((\bar{\BB}_X, d_H)\).

Indeed, \eqref{t3.14:e1} and \eqref{t3.14:e2} imply \(\bar{B}_m \supseteq \bar{B}\) for every \(m \in \mathbb{N}\). Consequently, we have
\begin{equation}\label{t3.14:e3}
d_H(\bar{B}_m, \bar{B}) = \diam (\bar{B}_m \cup \bar{B}) = \diam \bar{B}_m
\end{equation}
by Lemma~\ref{l6.12}. Now using~\eqref{e2.3} with \(A = \bar{B}_m\) and \(p_0 = c\) we see that \eqref{t3.14:e2} and \eqref{t3.14:e3} imply
\[
d_H(\bar{B}_m, \bar{B}) \leqslant \frac{1}{m}
\]
for every \(m \in \mathbb{N}\). Thus, the sequence \((\bar{B}_m)_{m \in \mathbb{N}}\) converges to \(\bar{B}\) in \((\bar{\BB}_X, d_H)\). It is clear that \(\bar{B}_m\) belongs to \(\bar{\BB}_X^0\) for each \(m \in \mathbb{N}\). Hence, \(\bar{\BB}_X^0\) is dense in \(\bar{\BB}_X\) as required.

Let \(\mathbf{A}\) be a dense discrete subset of \((\bar{\BB}_X, d_H)\). To complete the proof we only note that the equality
\begin{equation*}
\mathbf{A} = \bar{\BB}_X^0.
\end{equation*}
holds by Corollary~\ref{c1.13}.
\end{proof}

\section{Discreteness, completeness, compactness and separability of balleans}
\label{sec5}

Let us start from the discreteness of the ultrametric ballean \(\bar{\BB}_X\).

\begin{theorem}\label{t4.1}
Let \((X, d)\) be an ultrametric space. Then \((X, d)\) is discrete if and only if \((\bar{\BB}_X, d_H)\) is discrete.
\end{theorem}

\begin{proof}
By Definition~\ref{d1.23}, \((X, d)\) is discrete if and only if
\begin{equation}\label{t4.1:e1}
X = \operatorname{iso}(X).
\end{equation}
Analogously, \((\bar{\BB}_X, d_H)\) is discrete if and only if
\begin{equation}\label{t4.1:e2}
\bar{\BB}_X = \operatorname{iso}(\bar{\BB}_X).
\end{equation}
Using equalities~\eqref{s1:e5} and \eqref{s1:e6} with \((Y, \delta) = (X, d)\), \(S = X\) and, respectively, with \((Y, \delta) = (\bar{\BB}_X, d_H)\), \(S = \bar{\BB}_X\), we can rewrite~\eqref{t4.1:e1} as
\begin{equation}\label{t4.1:e3}
\operatorname{acc}(X) = \varnothing
\end{equation}
and, respectively, \eqref{t4.1:e2} as
\begin{equation}\label{t4.1:e4}
\operatorname{acc}(\bar{\BB}_X) = \varnothing.
\end{equation}
Equality~\eqref{c3.18:e1} implies the logical equivalence \(\eqref{t4.1:e3} \Leftrightarrow \eqref{t4.1:e4}\), that give us the equivalence \(\eqref{t4.1:e1} \Leftrightarrow \eqref{t4.1:e2}\), as required.
\end{proof}

\begin{example}\label{ex3.20}
Let \(X\) be a nonempty subset of \(\mathbb{R}^+\) and \(d\) be the restriction on \(X \times X\) of the Delhomm\'{e}---Laflamme---Pouzet---Sauer ultrametric \(d^+\). Then \((\bar{\BB}_X, d_H)\) is discrete if and only if \(0 \notin \operatorname{acc}(X)\).
\end{example}

The next result can be considered as a specialization of Theorem~\ref{t4.1}.

\begin{theorem}\label{t4.3}
Let \((X, d)\) be an ultrametric space. Then \((X, d)\) is locally finite if and only if \((\bar{\BB}_X, d_H)\) is locally finite.
\end{theorem}

\begin{proof}
Let \((X, d)\) be locally finite. We claim that \((\bar{\BB}_X, d_H)\) is also a locally finite ultrametric space.

Let \(\mathbf{A}\) be an arbitrary bounded subset of \((\bar{\BB}_X, d_H)\). To prove our claim, we must show that \(\mathbf{A}\) is finite,
\begin{equation}\label{t4.3:e1}
|\mathbf{A}| < \infty.
\end{equation}
It is enough to consider the case
\begin{equation}\label{t4.3:e2}
|\mathbf{A}| \geqslant 2.
\end{equation}
Let \(A\) be a subset of \(X\) defined by
\begin{equation}\label{t4.3:e3}
A := \bigcup_{\bar{B} \in \mathbf{A}} \bar{B}.
\end{equation}
Now, since~\eqref{t4.3:e2} holds, we can use Proposition~\ref{p3.7} to obtain
\begin{equation}\label{t4.3:e4}
\diam \mathbf{A} = \diam A = \diam B^*,
\end{equation}
where \(\diam \mathbf{A}\) is the diameter of the set \(\mathbf{A}\) in \((\bar{\BB}_X, d_H)\), \(\diam A\) is the the diameter of the set \(A\) in \((X, d)\) and \(B^* = B^*(A)\) is the smallest closed ball containing the set~\(A\).

It follows from~\eqref{t4.3:e3} that every \(\bar{B} \in \mathbf{A}\) is a subset of the ball \(B^*\). Since \(\mathbf{A}\) is bounded subset of \((\bar{\BB}_X, d_H)\), the ball \(B^*\) is a bounded subset of \((X, d)\) by~\eqref{t4.3:e4}. Since \((X, d)\) is locally finite and \(B^*\) is bounded, we have the inequality
\[
|B^*| < \infty.
\]
Moreover, the number of subsets of \(B^*\) is equal to \(2^{|B^*|}\). Thus, we have
\[
|\mathbf{A}| \leqslant 2^{|B^*|} < \infty.
\]
Inequality~\eqref{t4.3:e1} follows.

To complete the proof it is enough to note that every metric space which is isometric to a subspace of locally finite space is also locally finite. Hence, by Corollary~\ref{c3.7}, \((X, d)\) is locally finite, if \((\bar{\BB}_X, d_H)\) is locally finite.
\end{proof}

The following theorem is a refinement of Theorem~\ref{t4.1} for the case of metrically discrete \((X, d)\).

\begin{theorem}\label{t4.4}
Let \((X, d)\) be an ultrametric space. Then \((X, d)\) is metrically discrete if and only if \((\bar{\BB}_X, d_H)\) is metrically discrete.
\end{theorem}

\begin{proof}
Let \((X, d)\) be metrically discrete. Then there is \(k_0 > 0\) such that
\begin{equation}\label{t4.4:e1}
d(x, y) \geqslant k_0
\end{equation}
for all distinct \(x\), \(y \in X\). To prove that \((\bar{\BB}_X, d_H)\) is metrically discrete it is enough to show that the inequality
\begin{equation}\label{t4.4:e2}
d_H(\bar{B}_1, \bar{B}_2) \geqslant k_0
\end{equation}
holds for all distinct \(\bar{B}_1\), \(\bar{B}_2 \in \bar{\BB}_X\).

Let us consider arbitrary distinct balls \(\bar{B}_1\), \(\bar{B}_2 \in \bar{\BB}_X\). By Lemma~\ref{l6.12}, we have the equality
\begin{equation}\label{t4.4:e3}
d_H(\bar{B}_1, \bar{B}_2) = \diam (\bar{B}_1\cup \bar{B}_2).
\end{equation}
It follows from \(\bar{B}_1 \neq \bar{B}_2\) that we can find \(x_1 \in \bar{B}_1\) and \(x_2 \in \bar{B}_2\) such that \(x_1 \neq x_2\). Inequality~\eqref{t4.4:e1} with \(x = x_1\) and \(y = x_2\) and equality~\eqref{t4.4:e3} give inequality~\eqref{t4.4:e2},
\begin{align*}
d_H(\bar{B}_1, \bar{B}_2) & = \diam (\bar{B}_1\cup \bar{B}_2) \\
& = \sup\{d(x, y) \colon x, y \in \bar{B}_1\cup \bar{B}_2\} \\
& \geqslant d(x_1, x_2) \geqslant k_0.
\end{align*}

Let \((\bar{\BB}_X, d_H)\) be metrically discrete. To complete the proof it suffices to note that that any subspace of metrically discrete space is also metrically discrete. Hence, \((\hat{X}, d_H)\) is metrically discrete by Corollary~\ref{c3.7}, where
\[
\hat{X} := \bigl\{\{x\} \colon x\in X\bigr\}.
\]
Thus, \((X, d)\) is metrically discrete as an isometric copy of \((\hat{X}, d_H)\).
\end{proof}

\begin{example}\label{ex4.5}
Let \(X\) be a subset of \(\mathbb{R}^{+}\) such that \(0\in X\) and let \(d\) be the restriction of the Delhomm\'{e}---Laflamme---Pouzet---Sauer ultrametric \(d^+\) on \(X \times X\). Then the space \((\bar{\BB}_X, d_H)\) is discrete if and only if this space is metrically discrete.
\end{example}

\begin{theorem}\label{t4.5}
Let \((X, d)\) be a metric space. Then the following statements are equivalent:
\begin{enumerate}
\item \label{t4.5:s1} The metric \(d \colon X \times X \to \mathbb{R}^+\) is equidistant.
\item \label{t4.5:s2} The metric \(d_H \colon \bar{\BB}_X \times \bar{\BB}_X \to \mathbb{R}^+\) is equidistant.
\item \label{t4.5:s3} The equality
\begin{equation}\label{t4.5:e1}
\bar{\BB}_X = \hat{X} \cup \{X\}
\end{equation}
holds, where
\begin{equation}\label{t4.5:e1.1}
\hat{X} := \bigl\{\{x\} \colon x\in X\bigr\}
\end{equation}
and \(\{X\}\) is the one-point set whose only element is the set \(X\).
\end{enumerate}
\end{theorem}

\begin{proof}
\(\ref{t4.5:s1} \Rightarrow \ref{t4.5:s2}\). Let \(d \colon X \times X \to \mathbb{R}^+\) be equidistant. Then there is \(t_1 > 0\) such that
\begin{equation}\label{t4.5:e2}
d(x, y) = t_1
\end{equation}
for all distinct \(x\), \(y \in X\). Hence, every \(A \subseteq X\) is a closed bounded subset of \((X, d)\) and, for all distinct nonempty \(A\), \(B \subseteq X\), equalities~\eqref{e6.11} and \eqref{t4.5:e2} imply
\begin{equation}\label{t4.5:e3}
d_H(A, B) = t_1.
\end{equation}
Since \(\bar{\BB}_X\) is a subset of \(\mathfrak{M}_X\), we also have
\begin{equation}\label{t4.5:e4}
d_H(\bar{B}_1, \bar{B}_2) = t_1
\end{equation}
for all distinct \(\bar{B}_1\), \(\bar{B}_2 \in \bar{\BB}_X\). Thus, \(d_H \colon \bar{\BB}_X \times \bar{\BB}_X \to \mathbb{R}^+\) is equidistant.

\(\ref{t4.5:s2} \Rightarrow \ref{t4.5:s1}\). Let \(d_H \colon \bar{\BB}_X \times \bar{\BB}_X \to \mathbb{R}^+\) be equidistant. By Corollary~\ref{c3.7}, the space \((\hat{X}, d_H)\) is subspace \((\bar{\BB}_X, d_H)\). Hence,
\[
d_H \colon \hat{X} \times \hat{X} \to \mathbb{R}^{+}
\]
is equidistant as a restriction of \(d_H \colon \bar{\BB}_X \times \bar{\BB}_X \to \mathbb{R}^{+}\) on the set \(\hat{X} \times \hat{X}\). Moreover, \((X, d)\) and \((\hat{X}, d_H)\) are isometric by Corollary~\ref{c3.7}. Consequently, the metric \(d\) is also equidistant.

\(\ref{t4.5:s1} \Rightarrow \ref{t4.5:s3}\). Suppose that \(d\) is equidistant. Let us prove equality~\eqref{t4.5:e1}.

Let us first consider the case when \(X\) is a singleton,
\begin{equation}\label{t4.5:e5}
X = \{x_0\}.
\end{equation}
Then, using the definition of \(\bar{\BB}_X\) and equality~\eqref{t4.5:e5} we obtain the equality
\begin{equation}\label{t4.5:e6}
\bar{\BB}_X = \{\{x_0\}\}.
\end{equation}
Equalities~\eqref{t4.5:e1.1} and~\eqref{t4.5:e5} also give us the equality
\begin{equation}\label{t4.5:e7}
\hat{X} = \{\{x_0\}\}.
\end{equation}
From the definition of the set \(\{X\}\) and~\eqref{t4.5:e5} it follows that
\begin{equation}\label{t4.5:e8}
\{X\} = \{\{x_0\}\}.
\end{equation}
Equality~\eqref{t4.5:e1} follows from~\eqref{t4.5:e6}--\eqref{t4.5:e8},
\begin{equation}\label{t4.5:e9}
\hat{X} \cup \{X\} = \{\{x_0\}\} \cup \{\{x_0\}\} = \{\{x_0\}\} = \bar{\BB}_X.
\end{equation}

Let us consider the case when \(X\) contains at least two points, \(|X| \geqslant 2\). It directly follows from the definition of the closed balls that
\begin{equation}\label{t4.5:e10}
\bar{\BB}_X = \{\bar{B} \in \bar{\BB}_X \colon \diam \bar{B} > 0\} \cup \hat{X}.
\end{equation}
Hence, \eqref{t4.5:e1} holds if
\begin{equation}\label{t4.5:e11}
\{X\} = \{\bar{B} \in \bar{\BB}_X \colon \diam \bar{B} > 0\}.
\end{equation}
Since \(d\) is equidistant and \(|X| \geqslant 2\) holds, there exists the unique \(\bar{B} \in \bar{\BB}_X\) with \(\diam \bar{B} > 0\). Moreover, we have \(X \in \bar{\BB}_X\) and \(\diam X > 0\) for every bounded metric space \(X\) which contains at least two points. Thus, \eqref{t4.5:e11} holds.

Equality~\eqref{t4.5:e1} follows.

\(\ref{t4.5:s3} \Rightarrow \ref{t4.5:s1}\). Let equality~\eqref{t4.5:e1} hold. We must show that \(d\) is equidistant. If we have \(|X| \leqslant 2\), then \(d\) evidently is equidistant.

Let us consider the case \(|X| \geqslant 3\).

Let \(c_1\) be a point of \(X\) and \(x_1\), \(x_2\) be two arbitrary distinct points of \(X \setminus \{c_1\}\). We claim that
\begin{equation}\label{t4.5:e13}
d(c_1, x_1) = d(c_1, x_2).
\end{equation}
Indeed, if \eqref{t4.5:e13} does not hold and \(r_1 := d(c_1, x_1)\), \(r_2 := d(c_1, x_2)\), then the closed balls \(\bar{B}_{r_1}(c_1)\) and \(\bar{B}_{r_2}(c_1)\) are different,
\begin{equation}\label{t4.5:e14}
\bar{B}_{r_1}(c_1) \neq \bar{B}_{r_2}(c_1).
\end{equation}
Using~\eqref{t4.5:e14} we obtain the inequality
\begin{equation}\label{t4.5:e15}
|\{\bar{B} \in \bar{\BB}_X \colon \diam \bar{B} > 0\}| \geqslant 2,
\end{equation}
but from~\eqref{t4.5:e1} it follows that
\begin{equation}\label{t4.5:e16}
|\{\bar{B} \in \bar{\BB}_X \colon \diam \bar{B} > 0\}| = 1,
\end{equation}
which contradicts~\eqref{t4.5:e15}. Thus, \eqref{t4.5:e13} holds for all different \(x_1\), \(x_2 \in X \setminus \{c_1\}\).

Let \(r \in (0, \infty)\) satisfy
\begin{equation}\label{t4.5:e17}
d(c_1, x_1) = r = d(c_1, x_2).
\end{equation}
Since \(x_1\) and \(x_2\) are arbitrary distinct points of \(X \setminus \{c_1\}\), we only need to show that
\begin{equation}\label{t4.5:e18}
d(x_1, x_2) = r.
\end{equation}

Let us assume the contrary that \(d(x_1, x_2) \neq r\) and consider the balls \(\bar{B}_{r}(x_1)\) and \(\bar{B}_{r_3}(x_1)\) with \(r\) satisfying~\eqref{t4.5:e17} and
\begin{equation}\label{t4.5:e19}
r_3 := d(x_1, x_2).
\end{equation}
Then using \eqref{t4.5:e17} and \eqref{t4.5:e19} we obtain
\begin{equation}\label{t4.5:e20}
\bar{B}_{r}(x_1) \neq \bar{B}_{r_3}(x_1)
\end{equation}
and
\begin{equation}\label{t4.5:e21}
\diam \bar{B}_{r}(x_1) \neq 0 \neq \diam \bar{B}_{r_3}(x_1).
\end{equation}
Now, reasoning in the same way as in proving of equality~\eqref{t4.5:e17}, we can obtain that~\eqref{t4.5:e20} and \eqref{t4.5:e21} contradicts~\eqref{t4.5:e16}. Equality \eqref{t4.5:e18} follows. The metric \(d\) is equidistant as required.
\end{proof}

\begin{theorem}\label{t4.7}
Let \(d\) be an equidistant metric on a set \(X\). Then the following statements are equivalent:
\begin{enumerate}
\item \label{t4.7:s1} \(|X| = 1\) or \(X\) is an infinite set.
\item \label{t4.7:s2} The metric spaces \((X, d)\) and \((\bar{\BB}_X, d_H)\) are isometric.
\end{enumerate}
\end{theorem}

\begin{proof}
It was shown in the proof of Theorem~\ref{t4.5} that
\[
d(x_1, x_2) = d_H(\bar{B}_1, \bar{B}_2)
\]
holds for all distinct \(x_1\), \(x_2 \in X\) and all distinct \(\bar{B}_1\), \(\bar{B}_2 \in \bar{\BB}_X\) (see equalities~\eqref{t4.5:e2} and \eqref{t4.5:e4}). Consequently, \((X, d)\) and \((\bar{\BB}_X, d_H)\) are isometric if and only if
\begin{equation}\label{t4.7:e1}
|X| = |\bar{\BB}_X|.
\end{equation}
If \(X\) is a singleton, then \(\bar{\BB}_X\) is also a singleton by equality~\eqref{t4.5:e6}. Hence, \eqref{t4.7:e1} is true if \(|X| = 1\).

Let us consider the case when \(X\) is an infinite set. It follows directly from~\eqref{t4.5:e1.1} that
\begin{equation}\label{t4.7:e2}
|X| = |\hat{X}|.
\end{equation}
Consequently, \(\hat{X}\) is also infinite, that implies
\begin{equation}\label{t4.7:e3}
|\hat{X} \cup \{X\}| = |\hat{X}|.
\end{equation}
Now using~\eqref{t4.5:e1} and equalities~\eqref{t4.7:e2}, \eqref{t4.7:e3} we obtain~\eqref{t4.7:e1}. Thus, the implication \(\ref{t4.7:s1} \Rightarrow \ref{t4.7:s2}\) is true.

Let \(X\) be finite and \(|X| \geqslant 2\) hold. The Axiom of Regularity in Zermelo---Fraenkel set theory implies that there is no set \(X\) such that \(X \in X\) (see, for example, page~63 in \cite{Jec2003ST}). Consequently, for given \(X\), there is no \(x \in X\) such that \(\{x\} = \{X\}\). The last statement, equality~\eqref{t4.5:e1} and~\eqref{t4.7:e2} give us
\[
|\bar{\BB}_X| = |\hat{X}| + |\{X\}| = |\hat{X}| + 1 = |X| + 1 > |X|.
\]
Thus, \eqref{t4.7:e1} does not hold if statement \ref{t4.7:s1} us false. Hence, the implication \(\ref{t4.7:s2} \Rightarrow \ref{t4.7:s1}\) is also true.
\end{proof}

The following lemma will be used to describe the condition under which the ballean \(\bar{\BB}_X\) is complete.

\begin{lemma}\label{l3.7}
Let \((X, d)\) be an ultrametric space. Then the set \(\bar{\BB}_X\) is a closed subset of the space \((\mathfrak{M}_X, d_H)\).
\end{lemma}

\begin{proof}
By Proposition~\ref{p3.2}, the set \(\bar{\BB}_X\) is a subset of \(\mathfrak{M}_X\). The set \(\bar{\BB}_X\) is closed in \((\mathfrak{M}_X, d_H)\) iff
\begin{equation}\label{l3.7:e1}
\operatorname{acc}_{\mathfrak{M}_X}(\bar{\BB}_X) \subseteq \bar{\BB}_X
\end{equation}
holds. Let a sequence \((\bar{B}_n)_{n \in \mathbb{N}}\) of closed balls converge to \(A \in \mathfrak{M}_X\),
\begin{equation}\label{l3.7:e2}
\lim_{n \to \infty} d_H(A, \bar{B}_n) = 0.
\end{equation}
To prove~\eqref{l3.7:e1} it suffices to show that
\begin{equation}\label{l3.7:e3}
A \in \bar{\BB}_X.
\end{equation}
Let us do it. Equality~\eqref{l3.7:e2} implies
\[
\lim_{n \to \infty} d_H(\bar{B}_n, \bar{B}_{n+1}) = 0.
\]
The last limit relation and equality~\eqref{l6.12:e1} give us
\[
\lim_{n \to \infty} \diam(\bar{B}_n \cup \bar{B}_{n+1}) = 0,
\]
that implies
\begin{equation}\label{l3.7:e4}
\lim_{n \to \infty} \diam(\bar{B}_n) = 0.
\end{equation}
Now using~\eqref{e6.11} and \eqref{l3.7:e4} we can prove that \(A = \{a\}\) for some \(a \in X\). Consequently, \(A\) is the closed ball with the center \(a\) and the radius \(r = 0\). Membership relation~\eqref{l3.7:e3} follows.
\end{proof}

\begin{proposition}\label{p3.11}
Let \((X, d)\) be a metric space. Then \((\mathfrak{M}_X, d_H)\) is complete if and only if \((X, d)\) is complete.
\end{proposition}

For the proof see, for example, Theorem~10.7.2 \cite{Sea2007}.

Let us now move to the completeness of \(\bar{\BB}_X\).

\begin{theorem}\label{t3.13}
Let \((X, d)\) be an ultrametric space. Then the following statements are equivalent:
\begin{enumerate}
\item\label{t3.13:s1} The space \((X, d)\) is complete.
\item\label{t3.13:s2} The space \((\bar{\BB}_X, d_H)\) is complete.
\end{enumerate}
\end{theorem}

\begin{proof}
\(\ref{t3.13:s1} \Rightarrow \ref{t3.13:s2}\). Let \((X, d)\) be complete. Then \((\mathfrak{M}_X, d_H)\) is complete by Proposition~\ref{p3.11}. The set \(\bar{\BB}_X\) is closed in \((\mathfrak{M}_X, d_H)\) by Lemma~\ref{l3.7}. Hence, \((\bar{\BB}_X, d_H)\) is a complete space by Theorem~\ref{t1.10}.

\(\ref{t3.13:s1} \Rightarrow \ref{t3.13:s2}\). Let \((\bar{\BB}_X, d_H)\) be complete and let
\[
\hat{X} := \bigl\{\{x\} \colon x\in X\bigr\}.
\]
Corollary~\ref{c3.7} implies that \(\hat{X}\) is a closed subset of \((\bar{\BB}_X, d_H)\). Consequently, \((\hat{X}, d_H)\) is complete by Theorem~\ref{t1.10}. Now using Corollary~\ref{c3.7} we see that the spaces \((\hat{X}, d_H)\) and \((X, d)\) are isometric. Thus, \((X, d)\) is also complete.
\end{proof}

The next proposition is a reformulation of Theorem~3.1 from~\cite{Hen1999CATBOTHM}.

\begin{proposition}\label{p3.15}
Let \((X, d)\) be a metric space. Then \((\mathfrak{M}_X, d_H)\) is totally bounded if and only if \((X, d)\) is totally bounded.
\end{proposition}

The restriction of Proposition~\ref{p3.15} to the set \(\bar{\BB}_X\) gives us the next theorem.

\begin{theorem}\label{t3.16}
Let \((X, d)\) be an ultrametric space. Then the following statements are equivalent:
\begin{enumerate}
\item \label{t3.16:s1} The space \((X, d)\) is totally bounded.
\item \label{t3.16:s2} The space \((\bar{\BB}_X, d_H)\) is totally bounded.
\end{enumerate}
\end{theorem}

\begin{proof}
\(\ref{t3.16:s1} \Rightarrow \ref{t3.16:s2}\). Let \((X, d)\) be totally bounded. Then \((\mathfrak{M}_X, d_H)\) is totally bounded by Proposition~\ref{p3.15}. Now the inclusion \(\bar{\BB}_X \subseteq \mathfrak{M}_X\) and Proposition~\ref{p1.6} imply that \((\bar{\BB}_X, d_H)\) is totally bounded.

\(\ref{t3.16:s2} \Rightarrow \ref{t3.16:s1}\). Let \((\bar{\BB}_X, d_H)\) be totally bounded. By Corollary~\ref{c3.7}, the set
\[
\hat{X} := \bigl\{\{x\} \colon x\in X\bigr\}
\]
is a subset of \(\bar{\BB}_X\). Consequently, the space \((\hat{X}, d_H)\) is totally bounded by Proposition~\ref{p1.6}. Since \((X, d)\) and \((\hat{X}, d_H)\) are isometric, \((X, d)\) is also totally bounded.
\end{proof}

Theorem~\ref{t1.11}, Theorem~\ref{t3.13} and Theorem~\ref{t3.16} give us the next conclusion.

\begin{corollary}\label{c3.17}
Let \((X, d)\) be an ultrametric space. Then \((\bar{\BB}_X, d_H)\) is compact if and only if \((X, d)\) is compact.
\end{corollary}

Using Corollary~\ref{c3.17} we can also proved the following theorem.

\begin{theorem}\label{t5.14}
Let \((X, d)\) be an ultrametric space. Then the following statements are equivalent:
\begin{enumerate}
\item \label{t5.14:s1} The space \((X, d)\) is locally compact.
\item \label{t5.14:s2} The space \((\bar{\BB}_X, d_H)\) is locally compact.
\end{enumerate}
\end{theorem}

\begin{proof}
\(\ref{t5.14:s1} \Rightarrow \ref{t5.14:s2}\). Let \((X, d)\) be locally compact. We must show that \((\bar{\BB}_X, d_H)\) is locally compact.

The ultrametric space \((\bar{\BB}_X, d_H)\) is locally compact if, for every \(\bar{B}_{1} \in \bar{\BB}_X\), there is \(r = r(\bar{B}_{1}) > 0\) such that the set
\begin{equation*}
\mathbf{A}_r(\bar{B}_{1}) := \{\bar{B} \in \bar{\BB}_X \colon d_H(\bar{B}, \bar{B}_{1}) \leqslant r\}
\end{equation*}
is compact. The last statement follows from Definition~\ref{d2.22} because, in the space \((\bar{\BB}_X, d_H)\), the set \(\mathbf{A}_r(\bar{B}_{1})\) is the closed ball with the center \(\bar{B}_{1}\) and the radius \(r\).

If \(\bar{B}_{1}\) is an isolated point of \((\bar{\BB}_X, d_H)\), then the existence of desirable \(r\) follows from Definition~\ref{d1.21} with \((Y, \delta) = (\bar{\BB}_X, d_H)\) and \(s = \bar{B}_{1}\).

Suppose that \(\bar{B}_{1} \in \operatorname{acc}(\bar{\BB}_X)\). Then, by Corollary~\ref{c3.18}, there is \(c_1 \in \operatorname{acc}(X)\) such that the equality \(\bar{B}_{1} = \{c_1\}\) holds and, consequently,
\begin{equation}\label{t5.14:e3}
\mathbf{A}_r(\bar{B}_{1}) = \mathbf{A}_r(\{c_1\}) = \{\bar{B} \in \bar{\BB}_X \colon d_H(\bar{B}, \{c_1\}) \leqslant r\}.
\end{equation}
Let us consider the set
\begin{equation}\label{t5.14:e4}
A_r := \bigcup_{\bar{B} \in \mathbf{A}_r(\{c_1\})} \bar{B}.
\end{equation}
Then the equality
\begin{equation}\label{t5.14:e5}
A_r = \bar{B}_r(c_1)
\end{equation}
holds. Indeed, the inclusion
\begin{equation}\label{t5.14:e6}
\bar{B}_r(c_1) \subseteq A_r
\end{equation}
follows from~\eqref{t5.14:e3} and~\eqref{t5.14:e4} by Corollary~\ref{c3.7}. Let us prove the converse inclusion
\begin{equation}\label{t5.14:e7}
A_r \subseteq \bar{B}_r(c_1).
\end{equation}
The last inclusion holds iff the inequality
\begin{equation}\label{t5.14:e8}
d_H(\bar{B}, \{c_1\}) \leqslant r
\end{equation}
imply the inequality
\begin{equation}\label{t5.14:e9}
d(x, c_1) \leqslant r
\end{equation}
for all \(x \in \bar{B}\) and \(\bar{B} \in \mathbf{A}_r(\{c_1\})\).

Suppose that~\eqref{t5.14:e8} holds. Then, by Lemma~\ref{l6.12}, we have the inequality
\begin{equation}\label{t5.14:e10}
\diam (\{c_1\}, \bar{B}) \leqslant r
\end{equation}
for every \(\bar{B} \in \mathbf{A}_r(\{c_1\})\). Inequality~\eqref{t5.14:e10} and equality~\eqref{e2.3} with \(A = \{c_1\} \cup \bar{B}\) and \(p_0 = c_1\) give us inequality~\eqref{t5.14:e9}. Thus, \eqref{t5.14:e7} holds.

Equality~\eqref{t5.14:e5} follows from \eqref{t5.14:e6} and \eqref{t5.14:e7}.

Now we are ready to find \(r > 0\) such that the closed ball \(\mathbf{A}_r(\{c_1\})\) is a compact subset of \((\bar{\BB}_X, d_H)\). Let us do it.

Since \((X, d)\) is locally compact, there exists \(r_1 > 0\) such that the closed ball \(\bar{B}_{r_1}(c_1)\) is compact in \((X, d)\). Hence, the ultrametric space \((A_{r_1}, d|_{A_{r_1} \times A_{r_1}})\) is compact by equality~\eqref{t5.14:e5}. Let \(\bar{\BB}_{A_{r_1}}\) be the ballean of the space \((A_{r_1}, d|_{A_{r_1} \times A_{r_1}})\). Then the ultrametric space \((\bar{\BB}_{A_{r_1}}, d_H)\) is compact by Corollary~\ref{c3.17}. Corollary~\ref{c3.4} give us the equality
\[
\bar{\BB}_{A_{r_1}} = \mathbf{A}_{r_1}(\{c_1\}).
\]
Hence, the space \((\mathbf{A}_{r_1}(\{c_1\}), d_H)\) is also compact.

Thus, the ultrametric space \((\bar{\BB}_X, d_H)\) is locally compact as required.

\(\ref{t5.14:s2} \Rightarrow \ref{t5.14:s1}\). Let \((\bar{\BB}_X, d_H)\) be locally compact. By Corollary~\ref{c3.7}, the set
\[
\hat{X} := \bigl\{\{x\} \colon x\in X\bigr\}
\]
is closed subset of \((\bar{\BB}_X, d_H)\). Now using Theorem~\ref{t1.10.1} and Definition~\ref{d2.22} we can prove that \((\hat{X}, d_H)\) is also locally compact. Moreover, the spaces \((X, d)\) and \((\hat{X}, d_H)\) are isometric by Corollary~\ref{c3.7}. Thus, \((X, d)\) is locally compact as required.
\end{proof}

The following theorem shows that \((X, d)\) and \((\bar{\BB}_X, d_H)\) can only be boundedly compact together.

\begin{theorem}\label{t3.17}
Let \((X, d)\) be an ultrametric space. Then the following statements are equivalent:
\begin{enumerate}
\item \label{t3.17:s1} The space \((X, d)\) is boundedly compact.
\item \label{t3.17:s2} The space \((\bar{\BB}_X, d_H)\) is boundedly compact.
\end{enumerate}
\end{theorem}

\begin{proof}
\(\ref{t3.17:s1} \Rightarrow \ref{t3.17:s2}\). Let \((X, d)\) be boundedly compact and let \(\mathbf{A}\) be an arbitrary bounded and closed subset of the metric space \((\bar{\BB}_X, d_H)\). We must show that \(\mathbf{A}\) is compact in \((\bar{\BB}_X, d_H)\).

It is clear that \(\mathbf{A}\) is compact if \(\mathbf{A}\) is finite. Let us consider the case when the number of elements of \(\mathbf{A}\) is infinite.

As in~\eqref{p3.7:e1}, we denote by \(A\) the set \(\bigcup_{\bar{B} \in \mathbf{A}} \bar{B}\). Since \(\mathbf{A}\) is bounded in \((\bar{\BB}_X, d_H)\), the set \(A\) is bounded in \((X, d)\) by Proposition~\ref{p3.7}. Let \(B^* = B^*(A) \in \bar{\BB}_X\) be the smallest ball containing \(A\) and let
\[
d^* := d|_{B^* \times B^*}
\]
be the restriction of the ultrametric \(d\) on the set \(B^* \times B^*\). The closed ball \(B^*\) is a closed and bounded subset of \((X, d)\). Consequently, the ultrametric space \((B^*, d^*)\) is compact because \((X, d)\) is boundedly compact by supposition. Now using Corollary~\ref{c3.17} we see that \((\bar{\BB}_{B^*}, d_H^*)\) is compact, where \(\bar{\BB}_{B^*}\) is the ballean of closed balls of the ultrametric space \((B^*, d^*)\). It follows from Proposition~\ref{p2.5} and the definitions of the set \(A\) and the ball \(B^*\) that every \(\bar{B} \in \mathbf{A}\) belongs also to the ballean \(\bar{\BB}_{B^*}\). Hence, the inclusions
\begin{equation}\label{t3.17:e1}
\mathbf{A} \subseteq \bar{\BB}_{B^*} \subseteq \bar{\BB}_X
\end{equation}
hold. Since \(\mathbf{A}\) is a closed subset of \(\bar{\BB}_X\), double inclusion~\eqref{t3.17:e1} implies that \(\mathbf{A}\) is also closed in \((\bar{\BB}_{B^*}, d_H^*)\). By Theorem~\ref{t1.10.1}, the set \(\mathbf{A}\) is compact as a closed subset of the compact space \((\bar{\BB}_{B^*}, d_H^*)\). Now Proposition~\ref{p1.5} and \eqref{t3.17:e1} imply that \(\mathbf{A}\) is compact in \((\bar{\BB}_X, d_H)\).

\(\ref{t3.17:s2} \Rightarrow \ref{t3.17:s1}\). Let \((\bar{\BB}_X, d_H)\) be boundedly compact. The set
\[
\hat{X} := \bigl\{\{x\} \colon x\in X\bigr\}
\]
is isometric copy of \((X, d)\) by Corollary~\ref{c3.7}. Hence, \((X, d)\) is boundedly compact iff \((\hat{X}, d_H)\) is boundedly compact. Thus, to complete the proof it suffices to show that \((\hat{X}, d_H)\) is boundedly compact.

Let \(S \subseteq \hat{X}\) be bounded and closed in \(\hat{X}\). We must show that \(S\) is compact in \(\hat{X}\). Corollary~\ref{c3.7} implies that \(\hat{X}\) is a closed subset of \((\bar{\BB}_X, d_H)\). Hence, \(S\) is also closed in \((\bar{\BB}_X, d_H)\). Consequently, \(S\) is compact in the space \((\bar{\BB}_X, d_H)\) because \((\bar{\BB}_X, d_H)\) is boundedly compact. Proposition~\ref{p1.5} implies that \(S\) is a compact subset of \(\hat{X}\). Since \(S\) is an arbitrary closed and bounded subset of \(\hat{X}\), the space \((\hat{X}, d_H)\) is boundedly compact as required.
\end{proof}

\begin{example}\label{ex4.19}
Let \(X \subseteq \mathbb{R}^{+}\) be nonempty, \(0 \in X\) holds, and let \(d\) be the restriction of the Delhomm\'{e}---Laflamme---Pouzet---Sauer ultrametric \(d^+\) on \(X \times X\). Then \((\bar{\BB}_X, d_H)\) is boundedly compact if and only if the set \(X \cap [0, t]\) is finite for every \(t > 0\).
\end{example}

The next our goal is to describe conditions under which \((\bar{\BB}_X, d_H)\) is separable.

\begin{theorem}\label{t3.14}
Let \((X, d)\) be an ultrametric space. Then the following statements are equivalent:
\begin{enumerate}
\item\label{t3.14:s1} The set \(\bar{\BB}_X^0\) is countable.
\item\label{t3.14:s2} The space \((\bar{\BB}_X, d_H)\) is separable.
\end{enumerate}
\end{theorem}

\begin{proof}
\(\ref{t3.14:s1} \Rightarrow \ref{t3.14:s2}\). Let \(\bar{\BB}_X^0\) be countable. Theorem~\ref{t3.11} implies that
\[
\bar{\BB}_X^0 = \operatorname{iso}(\bar{\BB}_X).
\]
The set \(\operatorname{iso}(\bar{\BB}_X)\) is dense in \((\bar{\BB}_X, d_H)\) by Proposition~\ref{p3.21}. Hence, \((\bar{\BB}_X, d_H)\) is separable as required.

\(\ref{t3.14:s2} \Rightarrow \ref{t3.14:s1}\). Let \((\bar{\BB}_X, d_H)\) be separable. Then the set \(\operatorname{iso}(\bar{\BB}_X)\) is countable by Proposition~\ref{p1.22}. Theorem~\ref{t3.11} gives us the equality
\[
\bar{\BB}_X^0 = \operatorname{iso}(\bar{\BB}_X).
\]
Thus, \(\bar{\BB}_X^0\) is countable as required.
\end{proof}

If \((X, d)\) is an ultrametric space with separable \((\bar{\BB}_X, d_H)\), then \((X, d)\) also is separable by Lemma~\ref{l1.16} and Theorem~\ref{t3.2}. The next example shows that \((X, d)\) can be separable and have the non-separable \((\bar{\BB}_X, d_H)\).

\begin{example}\label{ex3.19}
Let \(\mathbb{Q}\) be the set of all rational numbers and \(d^+\) be the Delhomm\'{e}---Laflamme---Pouzet---Sauer ultrametric. Let \(X := \mathbb{R}^+ \cap \mathbb{Q}\) and \(d\) be the restriction of \(d^+\) on the set \(X \times X\). Let us consider the closed balls
\[
\bar{B}_{1} := \bar{B}_{r_1}(0) \text{ and } \bar{B}_{2} := \bar{B}_{r_2}(0)
\]
in the ultrametric space \((X, d)\) for arbitrary \(r_1\), \(r_2 \in \mathbb{R}^+\). Then the equalities
\[
\bar{B}_{1} = [0, r_1] \cap \mathbb{Q} \text{ and } \bar{B}_{2} = [0, r_2] \cap \mathbb{Q}
\]
hold. Consequently, \(\bar{B}_{1}\) and \(\bar{B}_{2}\) are equal if and only if \(r_1 = r_2\). Hence, the set \(\bar{\BB}_X\) is uncountable because the set \(\mathbb{R}^+\) is uncountable. Thus, \((\bar{\BB}_X, d_H)\) is non-separable by Theorem~\ref{t3.14}.
\end{example}

\begin{remark}\label{r4.16}
Example~\ref{ex3.19} can also be used as a counterexample to Theorem~19.3 from~\cite{Sch1985}.
\end{remark}

The separable complete metric spaces are also known as \emph{Polish metric spaces}. Using this concept, we can draw the following conclusion.

\begin{corollary}\label{c3.13}
Let \((X, d)\) be an ultrametric space. Then \((\bar{\BB}_X, d_H)\) is Polish if and only if \((X, d)\) is Polish and \(\bar{\BB}_X^0\) is countable.
\end{corollary}

\begin{proof}
It directly follows from Theorems~\ref{t3.13} and \ref{t3.14}.
\end{proof}

\section{Conclusion. Expected results and open questions}
\label{sec6}

Theorem~\ref{t3.13} claims that an ultrametric space \((X, d)\) is complete if and only if the ballean  \((\bar{\BB}_X, d_H)\) is complete. Let us now consider an ultrametric space \((X, d)\), which is not necessarily complete and denote by \((\tilde{X}, \tilde{d})\) the completion of \((X, d)\).

\begin{conjecture}\label{con6.1}
The ballean \((\bar{\BB}_{\tilde{X}}, \tilde{d}_H)\) of \((\tilde{X}, \tilde{d})\) is isometric to the completion of the ballean \((\bar{\BB}_X, d_H)\) for every ultrametric space \((X, d)\).
\end{conjecture}

\begin{conjecture}\label{con6.2}
Let \((X, d)\) and \((Y, \rho)\) be ultrametric spaces. Then the following statements are equivalent:
\begin{enumerate}
\item\label{con6.2:s1} The completions \((\tilde{X}, \tilde{d})\) and \((\tilde{Y}, \tilde{\rho})\) are isometric.
\item\label{con6.2:s2} The ultrametric spaces \((\bar{\BB}_{X}^{0}, d_H)\) and \((\bar{\BB}_{Y}^{0}, \rho_H)\) are isometric.
\end{enumerate}
\end{conjecture}

Theorem~\ref{t4.7} gives us the condition under which \((X, d)\) and \((\bar{\BB}_X, d_H)\) are isometric for equidistant \((X, d)\). The next question naturally arises.

\begin{question}\label{q6.3}
Let \((X, d)\) be an ultrametric space such that \((\bar{\BB}_X, d_H)\) and \((X, d)\) are isometric. Does it follows from this that \(d \colon X \times X \to \mathbb{R}^{+}\) is equidistant?
\end{question}

Let us consider now two isometric ultrametric spaces \((X, d)\) and \((Y, \rho)\). If \(F \colon X \to Y\) is an isometry of these spaces, then it is easy to prove that the mapping
\[
\bar{\BB}_X \ni \bar{B} \mapsto \{F(x) \colon x \in \bar{B}\} \in \bar{\BB}_Y
\]
is correctly defined isometry of the balleans \((\bar{\BB}_X, d_H)\) and \((\bar{\BB}_Y, \rho_H)\).

\begin{problem}\label{probl6.4}
Let \((X, d)\) and \((Y, \rho)\) be isometric ultrametric spaces. Find conditions under which for every isometry \(\mathbf{F} \colon \bar{\BB}_X \to \bar{\BB}_Y\) there is an isometry \(F \colon X \to Y\) such that
\begin{equation}\label{probl6.4:e1}
\mathbf{F}(\bar{B}) = \{F(x) \colon x \in \bar{B}\}
\end{equation}
for all \(\bar{B} \in \bar{\BB}_X\).
\end{problem}

\begin{remark}\label{r6.5}
If \((X, d)\) and \((Y, \rho)\) are isometric equidistant spaces and \(|X| \neq 1\), then, using Theorem~\ref{t4.5}, one can find an isometry \(\mathbf{F} \colon \bar{\BB}_X \to \bar{\BB}_Y\) such that for every isometry \(F \colon X \to Y\) equality~\eqref{probl6.4:e1} does not hold for some \(\bar{B} \in \bar{\BB}_X\).
\end{remark}

\section*{Funding}

The author was supported by grant 359772 of the Academy of Finland.

\end{document}